\documentclass[12pt]{amsart}

\usepackage{geometry}
\usepackage{setspace}
\usepackage{hyperref}
\usepackage{amssymb}
\usepackage{amsmath}
\usepackage{amsthm}
\usepackage{stmaryrd}
\usepackage{graphicx}
\usepackage{mathrsfs}
\usepackage[noabbrev,capitalise ]{cleveref}
\usepackage{adjustbox}
\usepackage{comment}

\usepackage{mathtools}

\usepackage{relsize}

\usepackage{imakeidx}
\makeindex
\usepackage[]{algorithm}
\usepackage[noend]{algpseudocode}

\usepackage{xcolor}

\definecolor{ygreen}{rgb}{0.20, 0.46, 0.11}
\definecolor{newgreen}{rgb}{0.13, 0.4, 0.11}
\definecolor{dblue}{rgb}{0.13, 0.13, 0.5}
\definecolor{dmaroon}{rgb}{0.55, 0.08, 0.24}
\definecolor{rorange}{rgb}{0.64, 0.06, 0.04}
\definecolor{pblue}{rgb}{0.28, 0.1, 0.82}
\definecolor{cblue}{rgb}{0.2,0.17,0.85}
\definecolor{ao}{rgb}{0.08, 0.5, 0.04}

\hypersetup{
    colorlinks=true,
    linkcolor=cblue,
    filecolor=magenta,      
    urlcolor=cblue,
    citecolor=cblue,
}
\usepackage[shortlabels]{enumitem}

\usepackage{quiver}

\input xy
\xyoption{all}
\usepackage[color,matrix,arrow]{xy}
\usepackage{tikz}
\usepackage{tikz-cd}
\usetikzlibrary{bending}
\usetikzlibrary{arrows.meta}
\tikzcdset{arrow style=tikz, diagrams={>=stealth}}
\usetikzlibrary{matrix,arrows,decorations.markings, calc, positioning}
\usetikzlibrary {shapes.geometric}
\usetikzlibrary {intersections}

\let\O=\relax

\let\epsilon=\relax

\newcommand{\epsilon}{\varepsilon}

\newcommand{\O}{\mathcal{O}}

\newcommand{\cO}{\O}

\theoremstyle{plain}
\newtheorem{theorem}{Theorem}[section]
\newtheorem{lemma}[theorem]{Lemma}
\newtheorem{proposition}[theorem]{Proposition}
\newtheorem{corollary}[theorem]{Corollary}

\newtheorem*{claim*}{Claim}

\theoremstyle{definition}
\newtheorem{definition}[theorem]{Definition}
\newtheorem*{definition*}{Definition}
\newtheorem*{lemma*}{Lemma}
\newtheorem{conjecture}[theorem]{Conjecture}
\newtheorem*{conjecture*}{Conjecture}
\newtheorem{example}[theorem]{Example}
\newtheorem{remark}[theorem]{Remark}

\newtheorem{notation}[theorem]{Notation}

\newcommand{\defn}[1]{\textsl{#1}} 

\DeclareMathOperator{\Tu}{Tu}
\DeclareMathOperator{\Sub}{Sub}
\DeclareMathOperator{\sub}{\Sub}
\DeclareMathOperator{\tr}{\Tr}
\DeclareMathOperator{\Tr}{Tr}
\DeclareMathOperator{\Res}{Res}
\DeclareMathOperator{\Conj}{Conj}
\DeclareMathOperator{\Comp}{Comp}

\DeclareMathOperator{\Refl}{Refl}
\DeclareMathOperator{\hull}{Hull}

\newcommand{\onto}{\twoheadrightarrow}

\newcommand{\Oa}{\O_a}
\newcommand{\M}{\mathrm{M}}
\newcommand{\Om}{\O_m}
\newcommand{\mc}[1]{\M({#1})}

\newcommand{\res}{{\leq}}
\newcommand{\resSuccess}{{\res^S}}
\newcommand{\resFailure}{{\res^F}}
\newcommand{\cover}{{\prec}}
\newcommand{\coverSuccess}{{\cover^S}}
\newcommand{\coverFailure}{{\cover^F}}

\newcommand{\NO}{N_\O}

\let\oldtocsection=\tocsection
\renewcommand{\tocsection}[2]{\oldtocsection{#1}{#2}}
\let\oldtocsubsection=\tocsubsection
\renewcommand{\tocsubsection}[2]{\quad\oldtocsubsection{#1}{#2}}

\usepackage{tkz-euclide}
\usepackage{xstring}

\def\splicelist#1{
\StrCount{#1}{,}[\numofelem]
\ifnum\numofelem>0\relax
     \StrBehind[\numofelem]{#1}{,}[\mylast]%
\else
    \let\mylast#1%
\fi
}
\newcommand{\myroundpoly}[3][thin,color=black]{
\splicelist{#2}
\foreach \vertex [remember=\vertex as \succvertex
    (initially \mylast)] in {#2}{
    \coordinate (\succvertex-next) at ($(\succvertex)!#3!90:(\vertex)$);
    \coordinate (\vertex-previous) at ($(\vertex)!#3!-90:(\succvertex)$);
    \draw[#1] (\succvertex-next) --  (\vertex-previous);
}
\foreach \vertex in {#2}{
    \tkzDrawArc[#1](\vertex,\vertex-next)(\vertex-previous)
}
}

\newcommand{\nocontentsline}[3]{}
\let\origcontentsline\addcontentsline
\newcommand\stoptoc{\let\addcontentsline\nocontentsline}
\newcommand\resumetoc{\let\addcontentsline\origcontentsline}

\title{Maximal compatibility of disklike $G$-transfer systems}

\begin{document}

\author[D.\ DeMark]{David DeMark}
\address{Department of Mathematics, Earlham College, Richmond, IN, U.S.A.}
\email{\href{mailto:demarda@earlham.edu}{demarda@earlham.edu}}
\urladdr{\url{https://demark.website}}

\author[M.\ A.\ Hill]{Michael A.~Hill}
\address{ School of Mathematics, University of Minnesota, Minneapolis, MN, U.S.A.}
\email{\href{mailto:mahill@umn.edu}{mahill@umn.edu}}

\author[Y.\ Kamel]{Yigal Kamel}
\address{ Department of Mathematics, University of Illinois at Urbana-Champaign, Urbana, IL, U.S.A.}
\email{\href{mailto:ykamel2@illinois.edu}{ykamel2@illinois.edu}}
\urladdr{\url{https://yigalkamel.web.illinois.edu}}

\author[N.\ Niu]{Nelson Niu}
\address{ Department of Mathematics, University of Washington, Seattle, WA, U.S.A.}
\email{\href{mailto:nsniu@uw.edu}{nsniu@uw.edu}}
\urladdr{\url{https://nelsonniu.com}}

\author[K.\ Stoeckl]{Kurt Stoeckl}
\address{ School of Mathematics and Statistics, The University of Melbourne, Melbourne, Victoria, Australia}
\email{\href{mailto:kstoeckl@student.unimelb.edu.au}{kstoeckl@student.unimelb.edu.au}}
\urladdr{\url{https://kstoeckl.github.io}}

\author[D.\ Van Niel]{Danika Van Niel}
\address{ Department of Mathematics \& Statistics, Binghamton University, Binghamton, NY, U.S.A.}
\email{\href{mailto:dvanniel@binghamton.edu}{dvanniel@binghamton.edu}}

\author[G.\ Yan]{Guoqi Yan}
\address{ Shanghai Center for Mathematical Sciences, Fudan University, Shanghai, China}
\email{\href{mailto:yanguoqi1993@outlook.com}{yanguoqi1993@outlook.com}}
\urladdr{\url{https://sites.google.com/view/guoqiyan}}

\begin{abstract}
Transfer systems are a combinatorial model for $N_{\infty}$-operads, which encode commutative structures in equivariant homotopy theory. Blumberg--Hill and Chan gave criteria for when two transfer systems are a compatible pair, meaning they encode the additive transfers and multiplicative norms of a ring-type structure.   In this paper, given a transfer system encoding an additive structure, we give explicit formulae for the maximal transfer system it is compatible with. Our formulae simplify for disklike transfer systems, which typically encode additive structures. Further, we prove that  (maximal) compatibility is functorial with respect to the inflation map induced by a quotient of groups, letting us compute maximal compatible transfer systems as inflations of connected transfer systems.
\end{abstract}

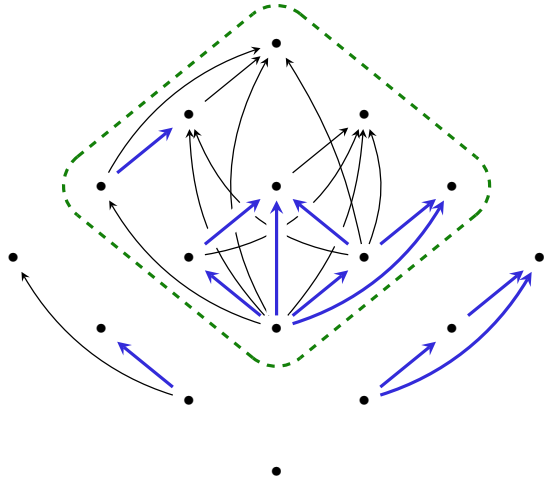
\maketitle

\centerline{
\hspace{1.0cm}
\begin{tikzcd}[ampersand replacement = \&, column sep = 7.5mm, row sep = 5.9mm,
every matrix/.append style = {name=m},
remember picture,nodes={scale=0.725}]
\&\&\& \bullet \\
\&\&\bullet \&\& \bullet\\
\&\bullet \&\& \bullet\&\& \bullet\&\&\\ 
{\bullet} \&\&\bullet \&\& \bullet\&\& {\bullet}\\ 
\&{\bullet} \&\& \bullet\&\& {\bullet}\&\&\\ 
\&\&{\bullet} \&\& {\bullet}\\
\&\&\& {\bullet}
\arrow[from=2-3, to=1-4, outer sep=-.5, overlay]
\arrow[from=4-5, to=1-4, bend right=10, overlay]
\arrow[from=3-4, to=2-5, outer sep=-.5, overlay]
\arrow[from=4-5, to=2-3, bend left, overlay]
\arrow[from=4-3, to=2-5, bend right, overlay]
\arrow[from=5-4, to=2-3, bend left=20, overlay]
\arrow[from=5-4, to=1-4, bend left, overlay]
\arrow[from=5-4, to=2-5, bend right=18, overlay]
\arrow[from=5-4, to=3-2, curve={height=-12pt}, overlay]
\arrow[from=4-5, to=2-5, bend right=20, overlay]
\arrow[from=3-2, to=2-3, very thick, crossing over, cblue, overlay]
\arrow[from=4-3, to=3-4, very thick, crossing over, cblue, overlay]
\arrow[from=5-4, to=3-6, curve={height=12pt}, very thick, crossing over, cblue, overlay]
\arrow[from=5-4, to=4-3, very thick, outer sep=-1.5, cblue, overlay]
\arrow[from=5-4, to=4-5, very thick, outer sep=-.5, cblue, overlay]
\arrow[from=5-4, to=3-4, very thick, crossing over, pos=.35, cblue, overlay]
\arrow[from=4-5, to=3-4, very thick, crossing over, cblue, overlay]
\arrow[from=5-6, to=4-7, very thick, cblue, overlay]
\arrow[curve={height=-12pt}, from=6-3, to=4-1, overlay]
\arrow[from=6-3, to=5-2, very thick, cblue, overlay]
\arrow[curve={height=12pt}, from=6-5, to=4-7, very thick, cblue, overlay]
\arrow[from=6-5, to=5-6, very thick, cblue, overlay]
\arrow[from=3-2, to=1-4, curve={height=-12pt}, overlay]
\arrow[from=4-5, to=3-6, very thick, cblue]
\end{tikzcd}
}
\begin{tikzpicture}[
remember picture, overlay]
\myroundpoly[ao, very thick, dashed]{m-1-4,m-3-6,m-5-4,m-3-2}{0.5cm};
\end{tikzpicture}

\vspace*{-0.7cm}

\section{Introduction}
Ring spectra are fundamental in algebraic topology, and highly structure ring spectra are central to the field. Here we consider associative or commutative ring spectra where the associativity and commutativity are encoded by specifying coherent homotopies witnessing these properties. In the commutative case, the multiplication is structured by an \(E_\infty\) operad, and May showed that all such \(E_\infty\) operads are equivalent \cite{mayEinfty}.

There is a similar story in the equivariant setting over a finite group $G$, with some subtleties. Here we can intertwine the action of \(G\) with the permutation action on the coordinates. Classically, this is no new structure: the defining feature of ``commutativity'' is that the order does not matter. In the homotopical context, since commutativity is extra structure encoded by explicit homotopies, we are no longer guaranteed to be able to do this. Unpacking these ideas and codifying these led to the notion of an \(N_\infty\) operad by Blumberg and the second author \cite{BHOperads}. 

The combined work of many authors, 
\cite{BHOperads, BalchinBarnesRoitzheim21Noperads, BonventrePereira21, GutierrezWhite18, Rubin21_combNinfty,  RubinDetectingOperads}, show that the \(\infty\)-category of $G$-$N_{\infty}$ operads is equivalent to a poset, the category of $G$-transfer systems. This allows us to study these complex $N_\infty$ operads using tangible combinatorial tools.

When viewed additively, the \(N_\infty\) operads structure connective \(G\)-spectra with certain transfers prescribed by the operad. When viewed multiplicatively, they parameterize ring spectra with certain norm maps, the algebraic shadow of which is the corresponding incomplete Tambara functors. Allowing both to vary introduces compatibility conditions arising from \(G\)-twisted versions of the distributive law of multiplication over addition, described by Blumberg and the second author and refined by Chan to transfer systems \cite{BiIncomplete, ChanTambara}. Moreover, for any fixed additive transfer system, there is a maximal compatible multiplicative transfer system. 

In this paper, we characterize compatible pairs $(\Oa,\Om)$ of transfer systems by devising formulae for the maximal compatible transfer system $\mc{\Oa}$ for any $G$-transfer system $\Oa$.
This characterizes all compatible pairs, as Blumberg--Hill observe that a pair of $G$-transfer systems $(\Oa,\O_m)$ are compatible, if, and only if, $\Om$ is contained by the unique maximal compatible transfer system $\mc{\Oa}$ \cite[Remark 7.86]{BiIncomplete}, see \Cref{prop:Maximal compatible characterises all compatibility}.
Our formulae simplify for \emph{disklike} transfer systems.
\begin{theorem}
[Theorems \ref{thm:r' can be assumed to be in M(O) 2} and \ref{thm:maximal compatible disk-like via cover relations}]\label{thm:main}
Let $\O$ be a $G$-transfer system.
The maximal compatible transfer system $\mc{\O}$ admits the following recursive formula
{
\begin{align*}
    \mc{\O} = \{ e \in \O \, | \, \forall r<e, r\in M(\mathcal O) \text{ and } r<^S e\},
\end{align*}
}where $r<e$ means, that in the poset of restrictions in $\O$, the transfer $e$ restricts into $r$, and $r<^S e$ means that this restriction satisfies the compatibility condition (\cref{def:compatibility successes and failures}).
Furthermore, if $\O$ is disklike, then we only need to check the cover relations $r\prec e$, yielding
{
\begin{align*}
    \mc{\O} = \{ e \in \O \, | \, \forall r\prec e, r \in \mc{\O} \text{ and $r\coverSuccess e$}\}.
\end{align*}}
\end{theorem}
We also apply
\cref{thm:main} to study the effect of inflation on (maximal) compatible transfer systems.
A surjection of groups $p:G\onto G/N$, where $N\trianglelefteq G$ and $p$ is the projection, induces the inflation functor on transfer systems $p^*:\Tr(G/N)\to \Tr(G)$ \cite{RubinOperadicLifts} (\cref{def:inflation}).
We show that inflation preserves all features relevant to (maximal) compatibility. 
\begin{theorem}
\label{thm:inflation}
    Let $N\trianglelefteq G$ be a normal subgroup. For $p:G\onto G/N$ the projection and $\O,\O_m$ a pair of $G/N$-transfer systems, the following statements are true.
    \begin{enumerate}[\emph{(\arabic*)}]
        \item \emph{(\cref{prop: inflation preserves disklike})} If $\O$ is disklike, then so is $p^*\O$.
        \item \emph{(\cref{prop: inflation preserves saturation})} If $\O$ is saturated, then so is $p^*\O$.
        \item \emph{(\cref{prop: inflation preserves compatibility})} If $(\O,\O_m)$ is compatible, then so is $(p^*\O,p^*\O_m)$.
        \item \emph{(\cref{prop: inflation preserves max compat})} If $(\O,\O_m)$ is maximally compatible, then so is $(p^*\O,p^*\O_m)$.
    \end{enumerate} 
\end{theorem}

One practical consequence of this theorem is a further refinement of \cref{thm:maximal compatible disk-like via cover relations}. 
\begin{corollary}[\cref{cor:universal-transfer}]
The computations of $\M(\O)$ for any disklike $G$-transfer systems $\O$ can be reduced to the computation of $\M(\O')$ for a disklike $Q$-transfer systems $\O'$ with the universal transfer $1\to Q$, where $Q$ is a quotient group of $G$.
\end{corollary}

The formulae of Theorems \ref{thm:r' can be assumed to be in M(O) 2} and \ref{thm:maximal compatible disk-like via cover relations} are straightforward to wield in practice, and are recursive in nature, which makes them well-suited to computational methods.
For example, \cref{alg:maximal compatible of disk-like} computes the maximal compatible transfer system of an arbitrary disklike $G$-transfer system $\Oa$ in time proportional to the number of restriction cover relations in $\Oa$ (\cref{prop:disk-like maximal compatible alg time complexity}).

The characterizations also reveal that saturation and maximal compatibility are intrinsically linked.
Namely we were able to prove the following facts using our formulae:
a transfer system is saturated if, and only if, it is self-compatible, this was originally proved in \cite{BiIncomplete} and we have another proof using our formulae (\cref{prop:saturated iff self compatible});
the maximal compatible transfer system is always saturated (\cref{prop:maximal compatible element is saturated}); 
the maximal compatible transfer system can only be disklike if it is saturated (\cref{prop: maximal compatible is disk-like implies self compatible});
and if $(\O,\Om)$ is compatible, then the saturated hull of $\Om$ remains compatible with $\O$ (\cref{cor:hull}).

It seems likely that the formulae identified in this paper will admit further simplifications, both for specific types of transfer systems, and for specific groups. For example, there is experimental evidence (\cref{prop:experimental verification of conj}), that if $\O$ is a disklike $G$-transfer system, then $\mc{\O}$ admits the following simplification of \cref{thm:main}, where we just check directly the compatibilty successes.

\begin{conjecture}[\cref{conj:maximal compatible disk-like via single restriction}]
    Let $\O$ be a disklike $G$-transfer system, then 
    \begin{align*}
        \mc{\O} = \{ e \in \O \, | \, \forall r < e, r <^S e \},
    \end{align*}
    where $r<e$ means, that in the poset of restrictions in $\O$, the transfer $e$ restricts into $r$, and $r<^S e$ means that this restriction satisfies the compatibility condition (\cref{def:compatibility successes and failures}).
\end{conjecture}

\subsection*{Acknowledgments}
Our work on this project was supported by the AMS Mathematics Research Communities Program in Homotopical Combinatorics and NSF grant DMS–1916439. 

DeMark, Hill, Niu, and Van Niel would like to thank the Isaac Newton Institute for Mathematical Sciences, Cambridge, and Queen's University Belfast for support and hospitality during the programme Operads and Calculus, where work on this paper was undertaken. This work was supported by EPSRC grant EP/Z000580/1.  
Hill, Kamel, and Van Niel would like to thank the Isaac Newton Institute for Mathematical Sciences, Cambridge, for support and hospitality during the programme New Horizons for Equivariance in Homotopy Theory, where work on this paper was undertaken. This work was also supported by EPSRC grant EP/Z000580/1.
Stoeckl was supported by an Australian Government Research Training Program (RTP) Scholarship and the Australian Research Council Future Fellowship FT210100256.
Van Niel was partially supported by the NSF grants DMS--2135884 and DMS--2052923.

\section{Background}\label{sec:Background}

\subsection{Transfer Systems}
We will denote the subgroup lattice of $G$ as $\Sub(G)$, which has a natural partial order $\leq$ defined by the subgroup inclusion. 

\begin{definition}\label{def:tranfsys} Let $G$ be a finite group. 
A \defn{$G$-transfer system} $\O$ is a 
binary relation refining subgroup inclusion on $\Sub(G)$,
consisting of \defn{transfers} denoted by $\to$, 
that satisfies the following additional conditions:
\begin{enumerate}
    \item(Reflexivity) $H\to H$ for all $H\leq G$, 
    \item (Conjugation) $K\to H$ implies $gKg^{-1}\to gHg^{-1}$ for all $g\in G$,
    \item (Restriction) $K\to H$ and $L\leq H$ implies $(K\cap L)\to (H\cap L)$, and
    \item (Composition) $L\to K$ and $K\to H$ implies $L\to H$.
\end{enumerate}
We denote the set of $G$-transfer systems by $\tr(G)$.
Transfer systems may also be defined on an arbitrary lattice by ignoring condition (2).
We call these \emph{categorical transfer systems} for clarity on the rare occasions they come up.
We denote the set of categorical transfer systems on $\mathcal P$ by $\tr(\mathcal P)$.
\end{definition}

Throughout this paper, a $G$-transfer system $\O$ will be represented as a directed graph whose vertices comprise $\Sub(G)$ and whose edges represent the partial order relation.  
\begin{notation}
We use arrows to denote transfers in $\O$.
We sometimes employ dotted lines to denote subgroup inclusions which may not be transfers in $\O$ for additional clarity.
For conciseness, we omit arrows that represent the reflexivity condition in both $G$-transfer systems and other binary relations known to be reflexive. 
\end{notation}
See the leftmost and rightmost diagrams of \cref{ex:Sigma3 example} for two examples of transfer systems.

There are two fundamental transfer systems for any group $G$, which are maximal and minimal among all $G$-transfer systems when ordered by inclusion.

\begin{definition}\label{def:complete and trivial transfers}
The \defn{complete $G$-transfer system} has \(H\to K\) if and only if \(H\leq K\).

The \defn{trivial $G$-transfer system} only has the transfers $H \to H$ for all $H \leq G$.
\end{definition}

The collection of all transfer systems for a given group forms a lattice under refinement.

\begin{proposition}[{\cite[Proposition 3.1]{RubinOperadicLifts}}] \label{prop:lattice of G-transfer systems}
The set of all $G$-transfer systems $\Tr(G)$ is a lattice ordered by containment.
If $\O$ and $\O'$ are two $G$-transfer systems, their \defn{meet} is $\O \land \O'$ is defined by intersection $\O \cap \O'$, and their \defn{join} $\O \lor \O'$ is the meet of all transfer systems which contain both.
\end{proposition}

One can also generate a transfer system from a binary relation.

\begin{definition}\label{def:generated tranfer system}
Let $B$ be a binary relation refining inclusion on $\Sub(G)$. The \defn{transfer system generated by $B$}, denoted $T(B)$, is defined to be the intersection of all $G$-transfer systems containing $B$.
We shall also say that $B$ \defn{generates} $T(B)$. 
\end{definition}

An equivalent perspective which we find more useful in our later calculations is closing the binary relation under reflexivity, conjugation, restriction, and then composition. Meaning, we first add in all missing reflexive arrows, conjugate arrows, restriction arrows, and then composition arrows. 

\begin{definition}[{\cite[Construction A.1]{RubinDetectingOperads}}]
Let $B$ be a binary relation on $\Sub(G)$ that refines inclusion. We have the following definitions.
\begin{enumerate}
    \item The \defn{reflexive closure} of $B$ is
        \[
        \Refl(B) = B\cup \bigcup_{H\leq G}\{ H \rightarrow H\}.
        \]
    \item The \defn{conjugate closure} of $B$ is
    \[
\Conj(B)= \bigcup_{K \to H \in B}\{gKg^{-1} \to gHg^{-1} \, | \, g\in G\}.
    \]
    \item The \defn{restriction closure} of $B$ is
    \[
    \Res(B)=  \bigcup_{K \to H \in B}\{L\cap K \to L|L \leq H\}.
    \]
    
    \item The \defn{composition closure} of $B$ is
    \begin{align*}
        \Comp(B)=   \{H_0 \to H_n \, | \,  H_0 \leq \dots \leq H_n \text{ with each } H_{i-1} \to H_{i} \in B\}.
    \end{align*}
\end{enumerate}
\end{definition}

\begin{proposition}[{\cite[Theorem A.2]{RubinDetectingOperads}}] \label{prop:transfer system generated by relations}
Let $B$ be a binary relation refining inclusion on $\Sub(G)$.
Then $T(B)$ is the closure of $B$ under, reflexivity, conjugation, restriction, then composition; or equivalently one may commute restriction and conjugation: 
\[
T(B) = \Comp(\Res(\Conj(\Refl(B)))) = \Comp(\Conj(\Res(\Refl(B)))).
\]
One may also commute $\Refl(-)$  with all other operators.
\end{proposition}

\begin{remark}
Although we treat reflexivity explicitly, one can also safely ignore it and just always assume all binary relations have reflexive edges.
\end{remark}

This process allows us to consider generating sets of edges which give a transfer system. 
We now provide a concrete example of a $G$-transfer system, and this method of generating a transfer system from a binary relation.

\begin{example}\label{ex:Sigma3 example}
The complete $\Sigma_3$-transfer system is displayed to the left below.
Recall that $\langle (12) \rangle,\langle (13) \rangle,$ and $\langle (23) \rangle$ are conjugate to each other.
To the right, we drop the subgroup labels for simplicity, and consider the transfer system generated by the binary relation on $\Sub(\Sigma_3)$ consisting of $\langle (12) \rangle \to \Sigma_3$.
We emphasize, that by \Cref{prop:transfer system generated by relations}, the transfer system generated by this relation can be calculated in two equivalent ways.


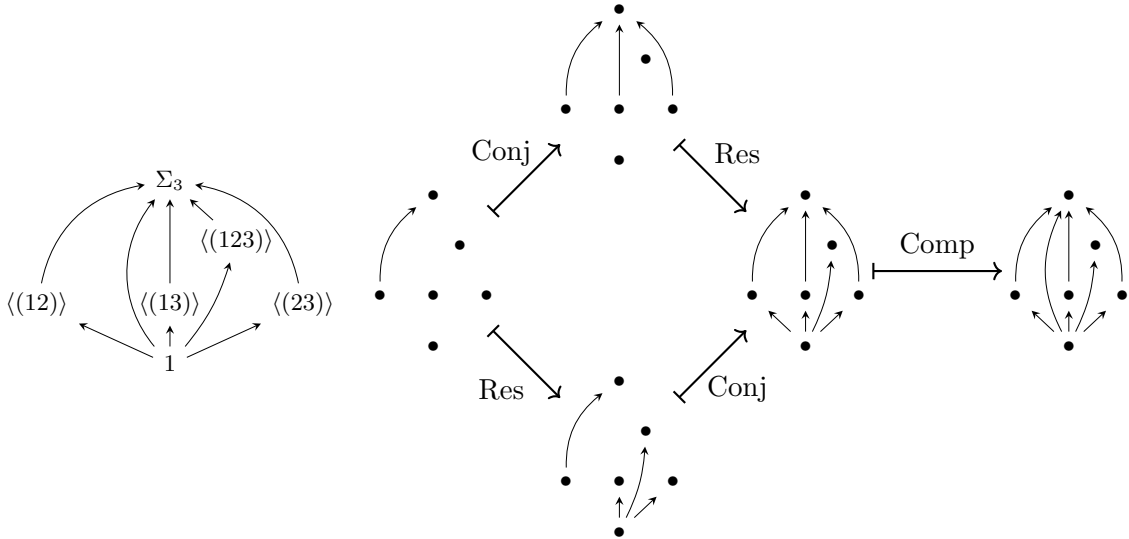
\newsavebox\aSigmaEx
\sbox\aSigmaEx{
{\centering\scriptsize
\hspace{0.0cm}
\begin{tikzcd}[ampersand replacement = \&, column sep = -0.8mm, row sep = small]
\&\& \bullet\\
\& \phantom{\bullet} \&\& \bullet\\
 \bullet \&\& \bullet  \&\& \bullet \\
\&\& \bullet
\arrow[from=3-1, to=1-3, bend left =27]
\end{tikzcd}
}
}

\newsavebox\bSigmaEx
\sbox\bSigmaEx{
{\centering\scriptsize
\hspace{0.0cm}
\begin{tikzcd}[ampersand replacement = \&, column sep = -0.8mm, row sep = small]
\&\& \bullet\\
\& \phantom{\bullet} \&\& \bullet\\
 \bullet \&\& \bullet  \&\& \bullet \\
\&\& \bullet
\arrow[from=3-1, to=1-3, bend left =27]
\arrow[from=3-3, to=1-3]
\arrow[from=3-5, to=1-3, bend right = 27]
\end{tikzcd}
}
}

\newsavebox\cSigmaEx
\sbox\cSigmaEx{
{\centering\scriptsize
\hspace{0.0cm}
\begin{tikzcd}[ampersand replacement = \&, column sep = -0.8mm, row sep = small]
\&\& \bullet\\
\& \phantom{\bullet} \&\& \bullet\\
 \bullet \&\& \bullet  \&\& \bullet \\
\&\& \bullet
\arrow[from=3-1, to=1-3, bend left =27]
\arrow[from=4-3, to=2-4, bend right = 11]
\arrow[from=4-3, to=3-3]
\arrow[from=4-3, to=3-5]
\end{tikzcd}
}
}

\newsavebox\dSigmaEx
\sbox\dSigmaEx{
{\centering\scriptsize
\hspace{0.0cm}
\begin{tikzcd}[ampersand replacement = \&, column sep = -0.8mm, row sep = small]
\&\& \bullet\\
\& \phantom{\bullet} \&\& \bullet\\
 \bullet \&\& \bullet  \&\& \bullet \\
\&\& \bullet
\arrow[from=3-1, to=1-3, bend left =27]
\arrow[from=3-3, to=1-3]
\arrow[from=3-5, to=1-3, bend right = 27]
\arrow[from=4-3, to=2-4, bend right = 11]
\arrow[from=4-3, to=3-1]
\arrow[from=4-3, to=3-3]
\arrow[from=4-3, to=3-5]
\end{tikzcd}
}
}

\newsavebox\eSigmaEx
\sbox\eSigmaEx{
{\centering\scriptsize
\hspace{0.0cm}
\begin{tikzcd}[ampersand replacement = \&, column sep = -0.8mm, row sep = small]
\&\& \bullet\\
\& \phantom{\bullet} \&\& \bullet\\
 \bullet \&\& \bullet  \&\& \bullet \\
\&\& \bullet
\arrow[from=3-1, to=1-3, bend left =27]
\arrow[from=3-3, to=1-3]
\arrow[from=3-5, to=1-3, bend right = 27]
\arrow[from=4-3, to=1-3, bend left = 27]
\arrow[from=4-3, to=2-4, bend right = 11]
\arrow[from=4-3, to=3-1]
\arrow[from=4-3, to=3-3]
\arrow[from=4-3, to=3-5]
\end{tikzcd}
}
}

\newsavebox\completeSigmaEx
\sbox\completeSigmaEx{
{\centering\scriptsize
\hspace{0.0cm}
\begin{tikzcd}[ampersand replacement = \&,column sep = -3mm, row sep = small]
\&\& {\Sigma_3} \\
\& \phantom{\langle (123) \rangle} \&\& {\langle (123) \rangle} \& \\
{\langle (12) \rangle} \&\& {\langle (13) \rangle} \&\& {\langle (23) \rangle} \\
\&\& 1
\arrow[from=2-4, to=1-3]
\arrow[from=3-1, to=1-3, bend left =35]
\arrow[from=3-3, to=1-3]
\arrow[from=3-5, to=1-3, bend right = 35]
\arrow[from=4-3, to=1-3, bend left = 40]
\arrow[from=4-3, to=2-4, bend right = 11]
\arrow[from=4-3, to=3-1]
\arrow[from=4-3, to=3-3]
\arrow[from=4-3, to=3-5]
\end{tikzcd}
}
}

\begin{center}
\vspace{0.5cm}
{\small
\begin{tikzpicture}[thick,node distance=3.5cm,inner sep = -0.4cm] 
\node(a) {\usebox{\aSigmaEx}}; 
\node(b) [above right of=a]{\usebox{\bSigmaEx}};
\node(c) [below right of=a]{\usebox{\cSigmaEx}}; 
\node(d) [below right of=b]{\usebox{\dSigmaEx}};
\node(e) [right of=d]{\usebox{\eSigmaEx}};
\node(complete) [left of=a]{\usebox{\completeSigmaEx}};
\draw [|->] (a) -- (b) node[midway,above left, yshift = 0.3cm,xshift = -0.3cm] {Conj};
\draw [|->] (a) -- (c)  node[midway,below left,yshift = -0.3cm,xshift = -0.3cm] {Res};
\draw [|->] (b) -- (d)  node[midway,above right,yshift = 0.3cm,xshift = 0.3cm] {Res};
\draw [|->] (c) -- (d)  node[midway,below right,yshift = -0.3cm,xshift = 0.3cm] {Conj};
\draw [|->, shorten <= 0.1cm,shorten >= 0.1cm] (d) -- (e)  node[midway,above, yshift =0.3cm] {Comp};
\end{tikzpicture}
}    
\vspace{0.25cm}
\end{center}

\noindent Note that the only edge which is not a transfer in the resulting transfer system is $\langle (123) \rangle \to \Sigma_3$.
\end{example}

For more examples of transfer systems, see the Hasse diagram of all ten $C_{pq}$-transfer systems in \cref{fig:Cpq maximal compatible pairs example}.
There are many useful properties that are observable in a transfer system. We will now define saturated and disklike transfer systems which as mentioned in the introduction have applications to certain varieties of ring spectra. Additionally, these properties have implications when we discuss compatible pairs of transfer systems in the next section and throughout this article.

\begin{definition}\label{def:saturated}
We say that a $G$-transfer system $\O$ is \defn{saturated} if for every sequence of subgroups $L \leq K \leq H \leq G$ whenever the transfer $L \to H$ is in $\O$ then $K \to H$ is also in $\O$. This can be visualized as the following:
\[\begin{tikzcd}
	L & K & H.
	\arrow[from=1-1, to=1-2]
	\arrow[curve={height=-12pt}, from=1-1, to=1-3]
	\arrow[dashed, from=1-2, to=1-3]
\end{tikzcd}\]
\noindent 
If $L \to H$ is in $\O$ and $K \to H$ is not in $\O$ this is called a \emph{saturation failure}.
 \end{definition}

Note that in the above diagram, since $\O$ is a transfer system, then $L\to K$ is in $\O$ by restriction.
If a transfer system is not saturated, then one can remedy the saturation failures to generate a saturated transfer system.

\begin{definition}\label{def:hull}
Let $\O$ be a $G$-transfer system.
The \emph{hull} of $\O$, denoted $\hull(\O)$, is the intersection of all saturated $G$-transfer systems which contain $\O$.
\end{definition}

It can be shown that $\hull(\O)$ can also be constructed by adding exactly the transfers which cause a saturation failure in $\O$.

\begin{definition}\label{def:disk-like}
A $G$-transfer system $\O$ is said to be \defn{disklike} if $\O$ can be generated by a set of transfers all with target $G$.
\end{definition}

We will refer to any such set of transfers to \(G\) which generate a disklike transfer system, say $\O$, as a \defn{set of disklike generators} for $\O$.
We say a set of disklike generators is \emph{maximal} if it contains all transfers $H \to G$ in $\O$. 
We say a set of disklike generators is \emph{minimal} if it is equal to the complexity of the $G$-transfer system it generates.
Within a disklike transfer system $\O$, we will refer to any transfer $H\to G\in \O$ as a disklike generator.

\begin{example}
The $\Sigma_3$-transfer system $\Oa$ of \cref{ex:Sigma3 example} is disklike.
The maximal set of disklike generators of $\Oa$ is $\{1 \to \Sigma_3, (12) \to \Sigma_3, (13) \to \Sigma_3, (23) \to \Sigma_3\}.$
There are three minimal sets of disklike generators of $\Oa$ these being $\{(12) \to \Sigma_3\},\{ (13) \to \Sigma_3\}$, and $\{(23) \to \Sigma_3\}$.
\end{example}

The definition of disklike can be simplified. To do this, we first need some lemmas.

\begin{lemma}[{\cite[Lemma A.6]{RubinDetectingOperads}}]\label{lem:disk-like generators closed under intersection}
Let $\O$ be a $G$-transfer system. If $H\to G$, and $K\to G$ are transfers in $\O$, then $H\cap K \to G$ is a transfer in $\O$.
\end{lemma}

Rubin showed that in the disklike case, the transfers are particularly simple to describe.

 \begin{proposition}[{\cite[Proposition A.7]{RubinDetectingOperads}}] \label{prop:disklikecharacterization}
 A $G$-transfer system $\O$ is disklike if and only if every transfer $e \in \O$ is a restriction of a transfer $H\to G \in \O$.
 \end{proposition}

There are transfer systems which are both saturated and disklike.
The complete and trivial $G$-transfer systems are both always saturated and disklike for any $G$. 
However, a transfer system being saturated or disklike does not imply the other. 
For example, the $\Sigma_3$-transfer system we obtain in \cref{ex:Sigma3 example} is disklike, but not saturated. 
The saturation failure is that the transfers $1 \to (123)$ and $1 \to \Sigma_3$ are in that transfer system, but the edge $(123) \to \Sigma_3$ is not.
Additionally, in \cref{fig:Cpq maximal compatible pairs example}, $C_{pq}$ has: $7$ saturated transfer systems; $7$ disklike transfer systems; and $4$ of those are both saturated and disklike.

\newsavebox\cpqa
\sbox\cpqa{
{\centering\scriptsize
\hspace{0.0cm}
\begin{tikzcd}[ampersand replacement = \&, column sep = 2.8mm, row sep = 2.2mm]
\& \bullet \\
\bullet \& \& \bullet \\
\& \bullet
\arrow[overlay, from=3-2, to=2-1, ]
\arrow[overlay, from=3-2, to=2-3, ]
\arrow[overlay, from=3-2, to=1-2, ]
\arrow[overlay, from=2-1, to=1-2, ]
\arrow[overlay, from=2-3, to=1-2, ]
\end{tikzcd}
}
}

\newsavebox\cpqb
\sbox\cpqb{
{\centering\scriptsize
\hspace{0.0cm}
\begin{tikzcd}[ampersand replacement = \&, column sep = 2.8mm, row sep = 2.2mm]
\& \bullet\\
\bullet \& \& \bullet \\
\& \bullet
\arrow[overlay, from=3-2, to=2-1, ]
\arrow[overlay, from=3-2, to=2-3, ]
\arrow[overlay, from=3-2, to=1-2, ]
\arrow[overlay, from=2-3, to=1-2, ]
\end{tikzcd}
}
}

\newsavebox\cpqc
\sbox\cpqc{
{\centering\scriptsize
\hspace{0.0cm}
\begin{tikzcd}[ampersand replacement = \&, column sep = 2.8mm, row sep = 2.2mm]
\& \bullet\\
\bullet \& \& \bullet\\
\& \bullet
\arrow[overlay, from=3-2, to=2-1, ]
\arrow[overlay, from=3-2, to=2-3, ]
\arrow[overlay, from=3-2, to=1-2, ]
\arrow[overlay, from=2-1, to=1-2, ]
\end{tikzcd}
}
}

\newsavebox\cpqd
\sbox\cpqd{
{\centering\scriptsize
\hspace{0.0cm}
\begin{tikzcd}[ampersand replacement = \&, column sep = 2.8mm, row sep = 2.2mm]
\& \bullet \\
\bullet \& \& \bullet\\
\& \bullet
\arrow[overlay, from=3-2, to=2-1, ]
\arrow[overlay, from=3-2, to=2-3, ]
\arrow[overlay, from=3-2, to=1-2, ]
\end{tikzcd}
}
}

\newsavebox\cpqe
\sbox\cpqe{
{\centering\scriptsize
\hspace{0.0cm}
\begin{tikzcd}[ampersand replacement = \&, column sep = 2.8mm, row sep = 2.2mm]
\& \bullet \\
\bullet \& \& \bullet \\
\& \bullet
\arrow[overlay, from=3-2, to=2-1, ]
\arrow[overlay, from=2-3, to=1-2, ]
\end{tikzcd}
}
}

\newsavebox\cpqf
\sbox\cpqf{
{\centering\scriptsize
\hspace{0.0cm}
\begin{tikzcd}[ampersand replacement = \&, column sep = 2.8mm, row sep = 2.2mm]
\& \bullet\\
\bullet \& \& \bullet \\
\& \bullet
\arrow[overlay, from=3-2, to=2-3, ]
\arrow[overlay, from=2-1, to=1-2, ]
\end{tikzcd}
}
}

\newsavebox\cpqg
\sbox\cpqg{
{\centering\scriptsize
\hspace{0.0cm}
\begin{tikzcd}[ampersand replacement = \&, column sep = 2.8mm, row sep = 2.2mm]
\& \bullet\\
\bullet \& \&\bullet\\
\& \bullet
\arrow[overlay, from=3-2, to=2-1,]
\arrow[overlay, from=3-2, to=2-3,]
\end{tikzcd}
}
}

\newsavebox\cpqh
\sbox\cpqh{
{\centering\scriptsize
\hspace{0.0cm}
\begin{tikzcd}[ampersand replacement = \&, column sep = 2.8mm, row sep = 2.2mm]
\& \bullet\\
\bullet \& \& \bullet\\
\& \bullet
\arrow[overlay, from=3-2, to=2-1,]
\end{tikzcd}
}
}

\newsavebox\cpqi
\sbox\cpqi{
{\centering\scriptsize
\hspace{0.0cm}
\begin{tikzcd}[ampersand replacement = \&, column sep = 2.8mm, row sep = 2.2mm]
\& \bullet\\
\bullet \& \& \bullet \\
\& \bullet
\arrow[overlay, from=3-2, to=2-3,]
\end{tikzcd}
}
}

\newsavebox\cpqj
\sbox\cpqj{
{\centering\scriptsize
\hspace{0.0cm}
\begin{tikzcd}[ampersand replacement = \&, column sep = 2.8mm, row sep = 2.2mm]
\& \bullet\\
\bullet \& \& \bullet\\
\& \bullet
\end{tikzcd}
}
}

\subsection{Compatible Pairs}
Let us recall the definition of a compatible pair of transfer systems.

\begin{definition}[{\cite[Definition 4.6]{ChanTambara}}] \label{def:compatible pair}
A pair of $G$-transfer systems $(\Oa,\Om)$ is \defn{compatible} if the following two conditions hold:
\begin{enumerate}
    \item $\Om \subseteq \Oa$, and
    \item for any $K, J \leq H \leq G$, if $K \to H$ is in $\Om$, and $K \cap J \to K$ is in $\Oa$, then $J \to H$ is in $\Oa$. 
\end{enumerate}

Condition $(2)$ can be visualized as \cref{eq:CompatDiag}, where if the thick edges are in $\Om$ (and consequently $\Oa$), and the thin edge is in $\Oa$, then the dashed edge must be in $\Oa$ to be compatible.
\begin{equation}\label{eq:CompatDiag}
\begin{tikzcd}[column sep = tiny, row sep = tiny]
& H \\
K && J \\
& {K \cap J}
\arrow[very thick, from=2-1, to=1-2]
\arrow[dashed, from=2-3, to=1-2]
\arrow[from=3-2, to=2-1]
\arrow[very thick, from=3-2, to=2-3]
\end{tikzcd}
\end{equation}
\end{definition}

In the above diagram, since $\Om$ is a transfer system, if $K \to H$ is in $\Om$ then restricting that transfer by $J$ shows that $K \cap J \to J$ must also be in $\Om$.

The most basic examples of compatible pairs are pairings with the complete and trivial transfer systems, see \cref{def:complete and trivial transfers}. The following is a well known result.

\begin{proposition}\label{prop:Complete and trivial compatibility}
    Let $C$ and $E$ be the complete and trivial $G$-transfer systems respectively. For any $G$-transfer system $\O$, $(C,\O)$ and $(\O, E)$ are always compatible pairs.
\end{proposition}

\begin{remark}
    A transfer system is not always compatible with itself. In fact, $(\O,\O)$ is a compatible pair if, and only if, $\O$ is saturated. 
\end{remark}

Compatibility is preserved under joins in the multiplicative slot.

\begin{proposition} [{\cite[Proposition 7.84]{BiIncomplete}}] \label{prop:join preserves compatibility}
Let $\Oa, \O_m$ and $\O'_m$ be $G$-transfer systems.
If $(\Oa, \O_m)$ and $(\Oa, \O'_m)$ are
both compatible, then $(\Oa, \O_m \lor \O'_m)$ is also compatible.
\end{proposition}

This shows that for any \(\O=\Oa\), there is a maximal compatible \(\Om\): the join of all compatible ones.

\begin{proposition} [{\cite[Corollary 7.85]{BiIncomplete}}] \label{prop:existence of maximal compatible pair}
For any $G$-transfer system $\O$, there exists a unique transfer system $\mc{\O}$ such that $(\O,\mc{\O})$ is maximally compatible.
\end{proposition}

Given a transfer system $\O$, $\mc{\O}$ also provides us with information about other compatible pairs $(\O,-)$.

\begin{proposition} [{\cite[Remark 7.86]{BiIncomplete}}] \label{prop:Maximal compatible characterises all compatibility}

A pair of $G$-transfer systems $(\Oa,\Om)$ is compatible if, and only if, $\Om \subseteq \mc{\Oa}$.

\end{proposition}

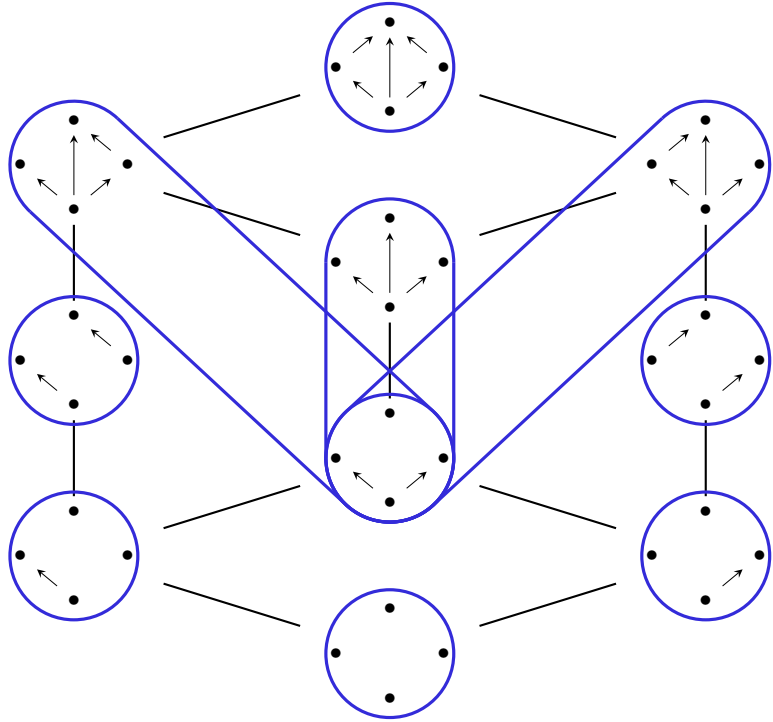
\begin{figure}[hbt!]
\begin{tikzpicture}[scale = 1, thick,node distance=2.6cm,inner sep = 0cm] 
\node(d) {\usebox{\cpqd}}; 
\node(g) [below of=d]{\usebox{\cpqg}}; 
\node(a) [above of=d]{\usebox{\cpqa}}; 
\node(j) [below of=g]{\usebox{\cpqj}}; 
\node(e) [at =($(d)!0.5!(g)$),xshift = -4.2cm] {\usebox{\cpqe}};
\node(f) [at =($(d)!0.5!(g)$),xshift = 4.2cm] {\usebox{\cpqf}};
\node(b) [at =($(a)!0.5!(d)$),xshift = -4.2cm] {\usebox{\cpqb}};
\node(c) [at =($(a)!0.5!(d)$),xshift = 4.2cm] {\usebox{\cpqc}};
\node(h) [at =($(g)!0.5!(j)$),xshift = -4.2cm] {\usebox{\cpqh}};
\node(i) [at =($(g)!0.5!(j)$),xshift = 4.2cm] {\usebox{\cpqi}};
\draw [thick] (a) to (b);
\draw [thick] (a) to (c);
\draw [thick] (b) to (d);
\draw [thick] (b) to (e);
\draw [thick] (c) to (d);
\draw [thick] (c) to (f);
\draw [thick] (d) to (g);
\draw [thick] (e) to (h);
\draw [thick] (f) to (i);
\draw [thick] (g) to (h);
\draw [thick] (g) to (i);
\draw [thick] (i) to (j);
\draw [thick] (h) to (j);
\myroundpoly[cblue, very thick]{b,g}{0.85cm};
\myroundpoly[cblue, very thick]{c,g}{0.85cm};
\myroundpoly[cblue, very thick]{d,g}{0.85cm};
\draw[cblue, very thick] (a) circle (0.85cm);
\draw[cblue, very thick] (e) circle (0.85cm);
\draw[cblue, very thick] (f) circle (0.85cm);
\draw[cblue, very thick] (g) circle (0.85cm);
\draw[cblue, very thick] (h) circle (0.85cm);
\draw[cblue, very thick] (i) circle (0.85cm);
\draw[cblue, very thick] (j) circle (0.85cm);
\end{tikzpicture}
\caption{
The Hasse diagram of the $C_{pq}$-transfer systems ordered by containment, $p$ and $q$ distinct primes.
Each transfer system $\O$ is paired with its maximal compatible transfer system $\mc{\O}$ via a blue circling.
}
\label{fig:Cpq maximal compatible pairs example}
\end{figure}

\section{Characterizing the Maximal Compatible Transfer System}\label{sec:max compatible}

\subsection{Maximal Compatible Formulae}

\begin{proposition}\label{prop:set expression of maximal compatible pair}
For any $G$-transfer system $\O$, a transfer \(K\xrightarrow{e}H\) is in \(\mc{\O}\) if and only if 
\[
\big(\O,T(e)\big)
\]
is compatible.
\end{proposition}

\begin{proof}
If $f\in \mc{\O}$ then $T(f) \subseteq \mc{\O}$ so $(\O,T(f))$ must be compatible. 
For readability, let $S = \{e \in \O \, | \, (\O, T(e)) \text{ is compatible}\}$.
In the other direction, we observe that 
\begin{align*}
    S \subseteq \underset{e \in S}{\bigvee} T(e).
\end{align*}
Since \Cref{prop:join preserves compatibility} shows that compatibility is closed under joins, then we know that $(\O,\underset{e \in S}{\bigvee} T(e))$ is a compatible pair, therefore $S \subseteq \underset{e \in S}{\bigvee} T(e) \subseteq \mc{\O}$. 
\end{proof}

When checking compatibility with \(T(e)\), it actually suffices to check a much smaller collection of edges.

\begin{lemma}\label{lem: Compatibility of conjugates}
    For any \(s\in\Oa\), if \(s\) is compatible with \(\Oa\), then so is any conjugate of \(s\).
\end{lemma}
\begin{proof}
    Assume \(s\) is compatible with \(\Oa\), and let \(g\in G\). Conjugating the diagram

\[
    \begin{tikzcd}[column sep = tiny, row sep = tiny]
	& H^g \\
	K^g & & J \\
	& {K^g \cap J}
	\arrow["{gsg^{-1}}", from=2-1, to=1-2]
	\arrow["\Oa"', draw={rgb,255:red,214;green,92;blue,92}, dashed, from=2-3, to=1-2]
	\arrow["\Oa", draw={rgb,255:red,214;green,92;blue,92}, from=3-2, to=2-1]
	\arrow[from=3-2, to=2-3]
    \end{tikzcd} 
\]
by \(g^{-1}\) yields the diagram
\[
    \begin{tikzcd}[column sep = tiny, row sep = tiny]
	& H \\
	K && g^{-1}Jg. \\
	& {K \cap g^{-1}Jg}
	\arrow["s", from=2-1, to=1-2]
	\arrow["\Oa"', draw={rgb,255:red,214;green,92;blue,92}, dashed, from=2-3, to=1-2]
	\arrow["\Oa", draw={rgb,255:red,214;green,92;blue,92}, from=3-2, to=2-1]
	\arrow[from=3-2, to=2-3]
    \end{tikzcd} 
    \]
Since \(s\) is compatible with \(\Oa\), we have that \(g^{-1}Jg\to H\) must be in \(\Oa\), and since \(\Oa\) is conjugation invariant, this gives \(J\to gHg^{-1}\) is in \(\Oa\) as desired.
\end{proof}

\begin{lemma}\label{lem: Compatibility of Composites}
    Given two arrows in \(\Oa\), \(I\xrightarrow{t} K\) and \(K\xrightarrow{s} H\), if \(s\) and \(t\) are both compatible with \(\Oa\), then so is the composite \(I\xrightarrow{s\circ t} H\).
\end{lemma}
\begin{proof}
    Consider the diagram
    
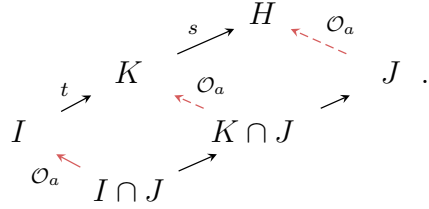
\[
    \begin{tikzcd}[column sep = tiny, row sep = tiny]
    &&  H\\
    & \phantom{J}K\phantom{J} && \phantom{J}J \phantom{J}. \\
    \phantom{J}I\phantom{J} && {K \cap J}  \phantom{J }\\
    & {I \cap J}
    \arrow["s", from=2-2, to=1-3]
    \arrow["t", from=3-1, to=2-2]
    \arrow[dashed, from=2-4, to=1-3,"\Oa"', draw={rgb,255:red,214;green,92;blue,92}]
    \arrow[dashed, from=3-3, to=2-2,"\Oa"', draw={rgb,255:red,214;green,92;blue,92}]
    \arrow[from=3-3, to=2-4]
    \arrow[from=4-2, to=3-1,"\Oa", draw={rgb,255:red,214;green,92;blue,92}]
    \arrow[from=4-2, to=3-3]
    \end{tikzcd}
    \]
    Since \(t\) is compatible with \(\Oa\), we have that \(K\cap J\to K\) is in \(\Oa\). Since \(s\) is compatible with \(\Oa\), the upper square shows that we have that \(J\to H\) is also in \(\Oa\).
\end{proof}

\begin{lemma}\label{lemma:checking compatibility on restriction of generator}
    Let $\Oa$ and $\Om$ be $G$-transfer systems, and $B_m$ a  binary relation such that $\Om = T(B_m)$. The pair $(\Oa, \Om)$ is compatible if and only if $\Res(B_m)$ is compatible with \(\Oa\). 
\end{lemma}

Note that $\Res(B_m)$ is not a transfer system in general.

\begin{proof}
One direction is immediate. For the other, we need to show that 
\[
    \Refl\big(\Comp\big(\Conj\big(\Res(B_m)\big)\big)\big)
\]
is compatible with \(\Oa\). The reflexive condition is immediate. Lemma~\ref{lem: Compatibility of Composites} and then Lemma~\ref{lem: Compatibility of conjugates} then show this is compatible with \(\Oa\) if 
\(
    \Res(B_m)\big),
\)
as desired.

\end{proof}

We can use this lemma to obtain the following improved formulae. 

\begin{proposition}\label{prop:computing Om via restricting edges}
    Let $\O$ be a $G$-transfer system.
    The maximal compatible transfer system can be written as the following set:
     \begin{gather}
         \label{eqn:M(O) formula}
         \begin{split}
     \mc{\O}= \{ K \to H \in \O \, \, |&\text{  for any }I < J \leq H,\\
     &\text{ if }K \cap I \to K \cap J \in \O\text{ then }I \to J \in \O\}.
     \end{split}
     \end{gather}
    Its complement, $\mc{\O}^c \coloneqq \O\backslash \mc{\O},$ can be written as:    
     \begin{gather*}
         \begin{split}
      \mc{\O}^c = \{ K \to H \in \O \, \, |&\text{  there exist }I < J \leq H\text{, such that }\\&K \cap I \to K \cap J \in \O\text{ and }I \to J \notin \O\}.
     \end{split}
     \end{gather*}
\end{proposition}

This proposition tells us that whenever we have a diagram of the following form in $\O$,
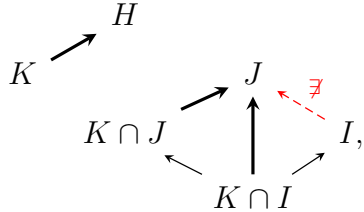
\[
\begin{tikzcd}[column sep = tiny, row sep = tiny]
&  H & &\\
K &  & J &  \\
 & {K\cap J} & & I, \\
& & {K \cap I}
\arrow[very thick, from=2-1, to=1-2]
\arrow[very thick, from=3-2, to=2-3]
\arrow[from=3-4, to=2-3, dashed,"{\not \exists}"', red]
\arrow[from=4-3, to=3-2]
\arrow[from=4-3, to=3-4]
\arrow[from=4-3, to=2-3, very thick]
\end{tikzcd}
\]
all of the bold edges are in $\mc{\O}^c$, as they all restrict into a compatibility failure.
See \cref{ex:non disk-like counter example to good conjecture} for a concrete realization of this diagram in the group $C_{p^2q}$.

\begin{proof}
From \cref{prop:set expression of maximal compatible pair}, we know that
\begin{align*}
     \mc{\O} = \{K\overset{e}{\to} H \in \O \, | \, (\O, T(e)) \text{ is a compatible pair}\}.
\end{align*}
\cref{lemma:checking compatibility on restriction of generator} allows us to check the compatibility of $(\O, T(e))$ by checking if $\Res(e)$ is compatible with \(\O\). 
Thus, 
 \begin{center}
      $ \mc{\O} = \{ K \to H \in \O \, | \,$  for any $I \leq J \leq H,$ if $K \cap I \to K \cap J \in \O$ then $I \to J \in \O$\}, 
\end{center}
and observe that if $I=J$ then we have $I\to J \in \O$, thus we only need to check for $I \lneq J$.
The complement can then be computed directly. 
\end{proof}

\begin{corollary}\label{cor:an edge with no restrictions is in the maximal compatible}
For any $G$-transfer system $\O$, if all restrictions of $e\in \O$ are reflexive, then $e\in \mc{\O}$.
\end{corollary}

\begin{proof}
Suppose there exists an transfer $e\colon K \to H \in \O$ such that all restrictions of $e$ are reflexive, and consider subgroups $I<J\leq H$.
As $e$ only admits reflexive restrictions, then $K\cap I = I$ and $K\cap J = J$. 
This implies that $I,J \leq K$.
By \Cref{eqn:M(O) formula}, the condition for $e$ to be in $\mc{\O}$ is that if $K \cap I = I \to K \cap J = J \in \O$, then $I \to J \in \O$, and this is vacuously true.
Therefore $e \in \mc{\O}$.
\end{proof}

Given the number of subgroups present in these formulae, we now introduce some notation that simplifies their presentation.

\begin{definition}\label{def:compatibility successes and failures}
Let $r$ and $e$ be transfers in a $G$-transfer system $\O$.
If $e$ restricts onto $r$, then we write $r \leq e$.
If $e$ non-reflexively restricts into $r$, then we write $r< e$.
If $r<e$ is a covering relation of the restriction poset, we write $r \prec e$. If $e:K \to H$, and $r:K\cap J \to J$ are transfers in $\O$ such that $r \leq e$, then we say that $r \leq e$ is a 
\begin{itemize}
    \item \defn{compatibility success}, denoted $r \resSuccess e$, if $K\cap J \to K \in \O$ implies $J\to H \in \O$,
    \item \defn{compatibility failure}, denoted $r \resFailure e$, if $K\cap J \to K \in \O$ and $J\to H \not \in \O$,
    \item \defn{covered compatibility success}, denoted $r \coverSuccess e$, if $r \resSuccess e$ and $r \prec e$, and
    \item \defn{covered compatibility failure}, denoted $r \coverFailure e$, if $r \resFailure e$ and $r \prec e$.
\end{itemize}

\end{definition}

For example, the expressions of \cref{prop:computing Om via restricting edges} simplify to
\begin{align}
    \mc{\O} &= \{ e \in \O \, | \, \forall r \leq e, \text{if } r' < r \text{ then $r' <^S r$}\}, \text{ and}\label{eq:expression for maximal compatible in terms of edges not subgroups}\\
    \mc{\O}^c &= \{ e \in \O \, |  \, \exists r \leq e, \text{ such that } \exists r' < r \text{ with $r' <^F r$}\}.\label{eq:expression for complement in terms of edges not subgroups}
\end{align}

We now refine our characterization of the maximal compatible transfer system in a manner that lets us recursively construct it using \cref{cor:an edge with no restrictions is in the maximal compatible}.
That is to say, starting with those transfers which restrict into no non-trivial transfers, and hence are trivially in the maximal compatible transfer system, we can construct the maximal transfer system from the bottom up.
Note that this recursion does terminate as \cref{eq:expression for maximal compatible in terms of edges not subgroups,eq:expression for complement in terms of edges not subgroups} contain strict inequalities, and every strict chain of inequalities in a finite poset terminates.

\begin{theorem}\label{thm:r' can be assumed to be in M(O) 2}
    Let $\O$ be a $G$-transfer system.
    The maximal compatible transfer system can be written as the following set 
    built by a recursion on $e$ starting with the transfers in $\O$ which do not restrict to any non-trivial transfers:
    \begin{align}\label{eq:max compatible with rhs in max compatible 2}
        \mc{\O} = \{ e \in \O \, | \, \forall r<e, r\in M(\mathcal O) \text{ and } r<^S e\},
    \end{align}
    where every restriction is in $\O$. 
    Its complement, $\mc{\O}^c \coloneqq \O\backslash \mc{\O},$ can be written as:
    \begin{align}\label{eq:complement with rhs in max compatible 2}
        \mc{\O}^c = \{ e \in \O \, | \, \exists r<e, r\not \in M(\mathcal O) \text{ or } r<^F e\}.
    \end{align}
\end{theorem}

\begin{proof}
Let $\mc{\O}$ be the formula of \cref{eq:expression for maximal compatible in terms of edges not subgroups}, and let $\mc{\O}'$ be the formula of \cref{eq:max compatible with rhs in max compatible 2}.
We will show these formulae coincide, which yields the equality of their complements by negation.
We first observe that $\mc{\O} \subseteq \mc{\O}'$ as the condition defining $\mc{\O}'$ is a priori weaker that of $\mc{\O}$.

In the other direction, let $e \in \mc{\O}'$, and let $r',r$ be arbitrary, non-reflexive restrictions in $\O$ such that $r'<r\leq e$.
If there exists no such non-trivial $r'$ then $e$ is vacuously in $\mc{\O}$ by \cref{cor:an edge with no restrictions is in the maximal compatible}.
Since $r'$ and $r$ are restrictions of $e$ then $r,r'\in \mc{\O}'$, or else $e \notin \mc{\O}'$ by definition of $\mc{\O}'$.
Then $r' <^S r$, as $r' < r$ and $r \in \mc{\O}'$.
Therefore $e \in \mc{\O}$ by \cref{eq:expression for maximal compatible in terms of edges not subgroups}, and $\mc{\O}' = \mc{\O}$. \end{proof}

\subsection{Consequences and Examples}

We now characterize maximal compatible transfer systems, and deduce further properties using the formulae of the previous section.

\begin{proposition}\label{prop:maximal compatible element is saturated}
For any $G$-transfer system $\O$, $\mc{\O}$ is always saturated.
\end{proposition}
\begin{proof}
If $\mc{\O}$ is not saturated, then it contains a transfer $K \to H$ such that there exists a subgroup $K < J < H$ with $J \to H \notin \mc{\O}$.
If $J \to H$ is not in $\O$ then $(\O, \mc{\O})$ is not compatible, to see this consider \ref{eq:CompatDiag} for the case when $K \cap J = K$.
Therefore $J \to H$ is in $\O$ and not in $\mc{\O}$, and therefore $J \to H \in \mc{\O}^c$.
\cref{prop:computing Om via restricting edges} then implies that $J \to H$ must restrict into a compatibility failure.
However, we observe in the following diagram of edges in $\O$ that this implies $K\to H \in \mc{\O}^c$ as it also restricts into a compatibility failure, a contradiction. 

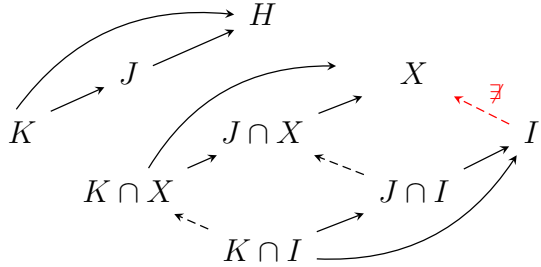
\[\begin{tikzcd}[column sep = tiny, row sep = tiny]
&&  H\\
& J  && \phantom{X\cap}X \phantom{X\cap}\\
\phantom{X\cap}K && {J \cap X} && I  \\
& {K \cap X} && { J \cap I}\\
&& {K \cap I}
\arrow[from=2-2, to=1-3]
\arrow[from=3-1, to=2-2]
\arrow[from=3-1, to=1-3, bend left]
\arrow[from=3-3, to=2-4]
\arrow[from=4-2, to=3-3]
\arrow[from=4-2, to=2-4, bend left]
\arrow[from=4-4, to=3-5]
\arrow[from=5-3, to=4-4]
\arrow[from=5-3, to=3-5, bend right]
\arrow[dashed, from=5-3, to=4-2, dashed]
\arrow[dashed, from=4-4, to=3-3, dashed]
\arrow[from=3-5, to=2-4, dashed,"\not \exists"', red]
\end{tikzcd}\]
Here $J \cap I \to J \cap X$ must be in $\O$ to have a compatibility failure, but since $\O$ is a transfer system, the edge $K \cap I \to K \cap X$ must also be in $\O$ by restriction. \end{proof}

The following is a well known result which we prove using $\M(-)$.

\begin{corollary}
\label{cor:hull}
    Let $(\Oa,\O_m)$ be a compatible pair of $G$-transfer systems, then $(\Oa,\hull(\O_m))$ is a compatible pair.
\end{corollary}
\begin{proof}
We observe that we must have $\hull(\O_m) \subseteq \mc{\Oa}$ as $\mc{\Oa}$ is saturated by \cref{prop:maximal compatible element is saturated}.
Thus \cref{prop:Maximal compatible characterises all compatibility} implies that $(\Oa,\hull(\O_m))$ must be a compatible pair.
\end{proof}

\begin{proposition}[{\cite[Corollary 7.77]{BiIncomplete}}] \label{prop:saturated iff self compatible}
A $G$-transfer system $\O$ is saturated, if, and only if, it is self compatible. 
\end{proposition}

\begin{example}\label{ex:non disk-like counter example to good conjecture}
Consider the following two $C_{p^2q}$-transfer systems.
The transfer system on the right is maximally compatible with the one on the left.
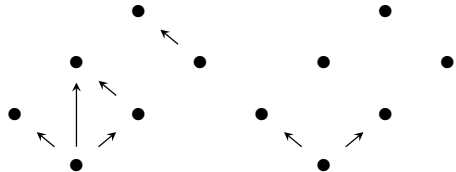
\[\begin{tikzcd}[column sep = tiny, row sep = tiny]
	&& \bullet &&&& \bullet \\
	& \bullet && \bullet && \bullet && \bullet \\
	\bullet && \bullet && \bullet && \bullet \\
	& \bullet &&&& \bullet
	\arrow[from=2-4, to=1-3]
	\arrow[from=3-3, to=2-2]
	\arrow[from=4-2, to=2-2]
	\arrow[from=4-2, to=3-1]
	\arrow[from=4-2, to=3-3]
	\arrow[from=4-6, to=3-5]
	\arrow[from=4-6, to=3-7]
\end{tikzcd}\]
If one were to add any edge to the transfer system on the right there would be a saturation failure, and therefore a compatibility failure.
\end{example}

In \cite{Hill-Meng-Nan}, it was shown that for \(C_{p^n}\), there is a maximal saturated transfer system in any transfer system and it is the maximal compatible transfer system. This is not true in general, as the following example shows.

\begin{example}\label{ex:hull preserves compatibility}
Consider the following disklike $C_{p^2q}$-transfer system $\Oa$.
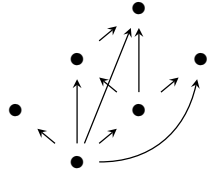
\[\begin{tikzcd}[column sep = tiny, row sep = tiny]
	&& \bullet \\
	& \bullet && \bullet \\
	\bullet && \bullet \\
	& \bullet
	\arrow[from=2-2, to=1-3]
	\arrow[from=3-3, to=1-3]
	\arrow[from=3-3, to=2-2]
	\arrow[from=3-3, to=2-4]
	\arrow[from=4-2, to=1-3]
	\arrow[from=4-2, to=2-2]
	\arrow[curve={height=18pt}, from=4-2, to=2-4]
	\arrow[from=4-2, to=3-1]
	\arrow[from=4-2, to=3-3]
\end{tikzcd}\]
For the following three transfer systems $(\O_L, \O_M, \O_R)$ we observe $\O_L$ is compatible with $\Oa$, $\O_M = \hull(\O_L)$ so is compatible with $\Oa$ by \cref{cor:hull},
and $\O_R$ is not compatible with $\Oa$ despite being the largest saturated transfer system contained by $\Oa$.
We display a red edge which is missing from $\Oa$, inducing a compatibility failure.

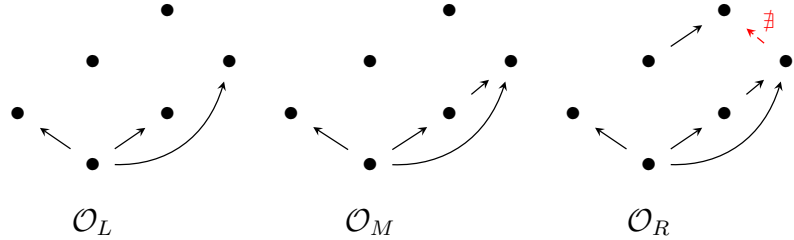
\[\begin{tikzcd}[column sep = tiny, row sep = tiny]
	&& \bullet &&&& \bullet &&&& \bullet \\
	& \bullet && \bullet && \bullet && \bullet && \bullet && \bullet \\
	\bullet && \bullet && \bullet && \bullet && \bullet && \bullet \\
	& \bullet &&&& \bullet &&&& \bullet \\
    & {\O_L} & & & & {\O_M} & & & & {\O_R}
	\arrow[from=2-10, to=1-11]
	\arrow[dashed,red,"\nexists",swap, from=2-12, to=1-11]
	\arrow[from=3-7, to=2-8]
	\arrow[from=3-11, to=2-12]
	\arrow[curve={height=18pt}, from=4-2, to=2-4]
	\arrow[from=4-2, to=3-1]
	\arrow[from=4-2, to=3-3]
	\arrow[curve={height=18pt}, from=4-6, to=2-8]
	\arrow[from=4-6, to=3-5]
	\arrow[from=4-6, to=3-7]
	\arrow[curve={height=18pt}, from=4-10, to=2-12]
	\arrow[from=4-10, to=3-9]
	\arrow[from=4-10, to=3-11]
\end{tikzcd}\]
In fact, $\O_M$ is the maximal compatible transfer system $\mc{\Oa}$. Given $\O_R$ is not compatible with $\Oa$, we observe that a maximal saturated transfer system contained in a transfer system $\O$ is not necessarily $\mc{\O}$.
\end{example}

Additionally, it turns out that if a maximal compatible transfer system is disklike, then the original transfer system must be self compatible, and hence saturated through \cref{prop:saturated iff self compatible}.

\begin{proposition}\label{prop: maximal compatible is disk-like implies self compatible}
Let $\O$ be a $G$-transfer system. If $\mc{\O}$ is disklike, then $\mc{\O} = \O$.
\end{proposition}

\begin{proof}
We will prove the contrapositive.
If $\mc{\O}\neq \O$, then \cref{thm:r' can be assumed to be in M(O) 2} implies there exist distinct transfers $r',r\in \O$ with $r'<^F r$ and $r' \in \mc{\O}$.
We will show that any transfer in $\O$ whose target is $G$ which restricts into $r'$, say $H_{r'}\to G$, must be an element of $\mc{\O}^c$ by \cref{thm:r' can be assumed to be in M(O) 2}. 
This will be sufficient to show $\mc{\O}$ cannot be disklike.

In the following diagram, the assumed missing transfer $J\to H$ of the compatibility failure $r'<^F r$, causes $H_{r'}\to G$ to also restrict into the compatibility failure $r'<^Fq$. 
Note that $(H_{r'}\cap H) \cap J = K\cap J$, and that the composite $K\cap J \to H$ of the failure $r'<^F r$ restricts into $K\cap J \to H_{r'}\cap H$.

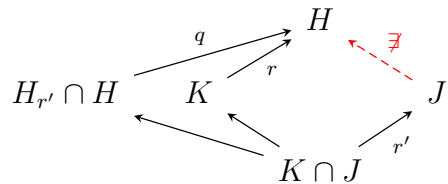
\begin{center}
\begin{tikzcd}[column sep = small, row sep = small]
&& H \\
H_{r'}\cap H & K & & J\\
&& K\cap J
\arrow[from=2-2, to=1-3, "r"']
\arrow[from=3-3, to=2-4, "r'"']
\arrow[from=3-3, to=2-2]
\arrow[from=3-3, to=2-1,]
\arrow[from=2-1, to=1-3, "q"]
\arrow[from=2-4, to=1-3,  dashed,"\not \exists"', red]
\end{tikzcd}
\end{center}
\end{proof}

We highlight one transfer system of particular interest.

\begin{definition}
    The \emph{tulip} $\Tu(G)$ is the $G$-transfer system that has all possible transfers that do not connect to $G$ and the reflexive transfer $G \to G$.
\end{definition}

\begin{example}
    The name comes from the following picture of $\Tu(Q_8)$:
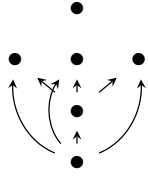
\[\begin{tikzcd}[column sep = tiny, row sep = tiny]
    & \bullet \\
    \bullet & \bullet & \bullet \\
    & \bullet \\
    & \bullet
    \arrow[from=3-2, to=2-1]
    \arrow[from=3-2, to=2-2]
    \arrow[from=3-2, to=2-3]
    \arrow[curve={height=-12pt}, from=4-2, to=2-1]
    \arrow[curve={height=-12pt}, from=4-2, to=2-2]
    \arrow[curve={height=12pt}, from=4-2, to=2-3]
    \arrow[from=4-2, to=3-2]
\end{tikzcd}\]
\end{example}

Note that the tulip is always saturated, furthermore the tulip is only disklike when $G$ has no non-trivial, proper subgroups, and is hence a cyclic group of prime order.

\begin{proposition}
    Let $\O$ be a $G$-transfer system which is not complete. If $\O$ contains $\Tu(G)$, then $\Tu(G) = \mc{\O}$.
\end{proposition}

\begin{proof}
Since the tuplic contains all non-trivial transfers with target a proper subgroup, so does \(\O\). This means the compatibility condition is immediate: \(\O\) is indistinguishable from the complete transfer system restricted to proper subgroups.

If we add any transfer \(H\to G\) to \(\Tu(G)\), then by transitivity, we have \(\{e\}\to G\) in the new system. The saturated hull of any transfer system containing the map from the minimal element to the maximal one is the complete transfer system.
\end{proof}

\subsection{Disklike Maximal Compatible Formulae}

Recall from \cref{def:compatibility successes and failures} that if $r' < r$ is a covering relation of the restriction poset, we write $r' \prec r$.
We will show that if $\O$ is a disklike transfer system we may refine the formulae of \cref{thm:r' can be assumed to be in M(O) 2} to check only covering relations instead of all relations.

\begin{lemma}\label{lem: induction so don't need notin}
If there is a transfer $s \in \O$ such that $s \notin \mc{\O}$, then there exist transfers $e \notin \mc{\O}$ and $r \in \mc{\O}$ such that $r <^F e \leq s$.
\end{lemma}

\begin{proof}
    Since $s \in \O$ and $s \notin \mc{\O}$, then $s \in \mc{\O}^c$.
    Meaning, there exists a transfer $s_1 < s$ in $\O$ such that $s_1 \notin \mc{\O}$ or $s_1 <^F s$.
    If $s_1\in \mc{\O}$, then this implies we must have $s_1 <^F s$, thus we can let $r = s_1$ and $e = s$ and we are done.
    If $s_1 \not \in \mc{\O}$, then given $s_1<s$, we can repeat this argument to find a transfer $s_2 < s_1$ such that $s_2 \notin \mc{\O}$ or $s_2 <^F s_1$.
    This induction argument will finish because $\mc{\O}$ contains all transfers in $\O$ which have no non-trivial restrictions.
    Therefore there exists some $n$ such that $s_n < s_{n-1} < \ldots < s_1 < s_0$, where $s_n \in \mc{\O}$, $s_{n-1} \notin \mc{\O}$, and $s_n <^F s_{n-1}$.
\end{proof}

Since every relation is generated by covering relations, we might ask if we can further simplify to checking only covering relations. In general, this is false, but for disklike transfer systems, it holds true.

\begin{theorem}\label{thm:maximal compatible disk-like via cover relations}
Let $\O$ be a disklike $G$-transfer system.
The maximal compatible transfer system can be written as:
\begin{align}\label{eq:disklike prec max compatible with rhs in max compatible}
    \mc{\O} = \{ e \in \O \, | \, \forall r\prec e, r \in \mc{\O} \text{ and $r\coverSuccess e$}\}.
\end{align}
Its complement can be written as:
\begin{align}\label{eq:disklike prec max compatible complement}
    \mc{\O}^c = \{ e \in \O \, | \, \exists r\prec e, \text{such that $r \not\in \mc{\O}$ or $r \coverFailure e$}\}.
\end{align}
\end{theorem}

\begin{proof}
Let $\mc{\O}$ be the formula of \cref{eq:max compatible with rhs in max compatible 2}, and let $\mc{\O}'$ be the formula of \cref{eq:disklike prec max compatible with rhs in max compatible}.
We will show the formulae of their complements coincide, which yields the equality of them by negation.
We first observe that $\mc{\O} \subseteq \mc{\O}'$ as the condition of $\mc{\O}'$ is a subset of the condition of $\mc{\O}$, which implies ${\mc{\O}'}^c \subseteq \mc{\O}^c$.

For the converse we will show that $\mc{\O}^c \subseteq {\mc{\O}'}^c$. 

First, note that we have a version of \cref{lem: induction so don't need notin} for \({\mc{\O}'}^c\): an edge \(s\) is in \({\mc{\O}'}^c\) if there is an edge \(r\in \mc{\cO}\) with
\[
r\prec^F e\leq s.
\]

We can use \cref{lem: induction so don't need notin} to say that without loss of generality we can consider $e \in \mc{\O}^c$ and $r \in \mc{\O}$ to be such that $r <^F e$:

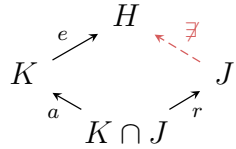
\[\begin{tikzcd}[column sep = tiny, row sep = tiny]
	& H \\
	K && J \\
	& {K \cap J}
	\arrow["e", from=2-1, to=1-2]
	\arrow["\not \exists"',color={rgb,255:red,214;green,92;blue,92}, dashed, from=2-3, to=1-2]
	\arrow["a", from=3-2, to=2-1]
	\arrow["r"', from=3-2, to=2-3]
\end{tikzcd}\]
where the red dashed edge is not in $\O$ since $r <^F e$.
There exists a chain of covering relations $r \prec r_1 \prec \ldots \prec r_n \prec e$.
Let us consider the following diagram:

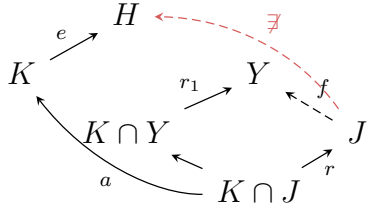
\[\begin{tikzcd}[column sep = tiny, row sep = tiny]
	& H \\
	K && Y \\
	& K\cap Y && J \\
	&& {K \cap J}
	\arrow["e", from=2-1, to=1-2]
	\arrow["{r_1}", from=3-2, to=2-3]
	\arrow["\not \exists"',color={rgb,255:red,214;green,92;blue,92}, curve={height=18pt}, dashed, from=3-4, to=1-2]
	\arrow["f"', dashed, from=3-4, to=2-3]
	\arrow["a", curve={height=-18pt}, from=4-3, to=2-1]
	\arrow[from=4-3, to=3-2]
	\arrow["r"', from=4-3, to=3-4]
\end{tikzcd}\]
where the transfer $a$ restricts to $K \cap J \to K\cap Y$.
If $f \notin \O$ then we are done as $r \prec^F r_1\leq e$.
So we can assume that $f \in \O$.

Since $e,f \in \O$ and $\O$ is disklike, then there must exist transfers $H_e \xrightarrow{\tilde{e}} G, H_f \xrightarrow{\tilde{f}} G$ in $\O$ which restrict to $e$ and $f$, respectively.
Below is the same diagram as above but with the subgroups written in terms of $Y,H,H_e,$ and $H_f$:
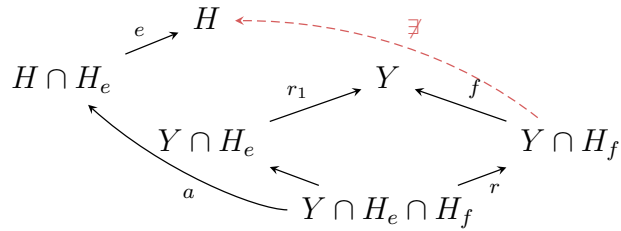
\[\begin{tikzcd}[column sep = tiny, row sep = tiny]
	& H \\
	{H \cap H_e} && Y \\
	& {Y \cap H_e} && {Y \cap H_f} \\
	&& {Y \cap H_e \cap H_f}
	\arrow["e", from=2-1, to=1-2]
	\arrow["{r_1}", from=3-2, to=2-3]
	\arrow["\not \exists"',color={rgb,255:red,214;green,92;blue,92}, curve={height=18pt}, dashed, from=3-4, to=1-2]
	\arrow["f"', from=3-4, to=2-3]
	\arrow["a", curve={height=-18pt}, from=4-3, to=2-1]
	\arrow[from=4-3, to=3-2]
	\arrow["r"', from=4-3, to=3-4]
\end{tikzcd}\]
note that since $r_1 < e$ and $Y < H$ then $X = H_e \cap H \cap Y = H_e \cap Y$.
Additionally, we can restrict 
\(\tilde{f}\) to \(H\), 
\(a\) to \(H_f \cap  H_e\), and \(\tilde{e}\) to \(H\cap H_f\). 
This yields the following diagram:
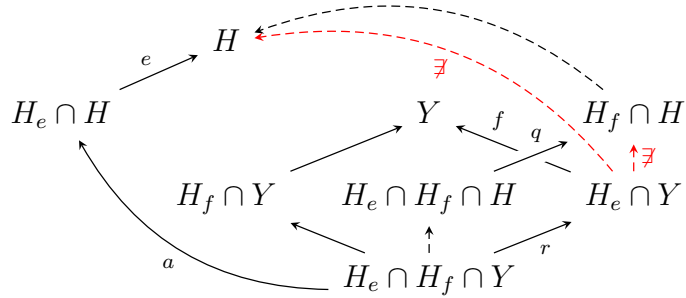
\[
\begin{tikzcd}[column sep = small, row sep = small]
{} & {H} & {} & {} \\
{H_e \cap H} & {} & {Y} & {H_f \cap H} \\
{} & {H_f \cap Y} & {H_e\cap H_f\cap H} & {H_e\cap Y}\\
{} & {} & {H_e\cap H_f\cap Y} & {}
\arrow[from=2-1, to=1-2, "e"]
\arrow[from=3-2, to=2-3]
\arrow[from=3-4, to=1-2, dashed,"{\not \exists}", red, bend right]
\arrow[from=4-3, to=3-4, "r"']
\arrow[from=3-4, to=2-3, "f"', near end]
\arrow[from=4-3, to=3-2]
\arrow[from=4-3, to=2-1, "a", bend left]
\arrow[from=4-3, to=3-3, dashed]
\arrow[from=3-3, to=2-4, crossing over, "q", near end]
\arrow[from=3-4, to=2-4, dashed,"{\not \exists}"', red]
\arrow[from=2-4, to=1-2, dashed, bend right]
\end{tikzcd}
\]

Since \(Y\cap H_f\to H\) is not in \(\O\), we cannot have \(Y\cap H_f\to H\cap H_f\) in \(\O\). This gives us a new, simplified diagram:

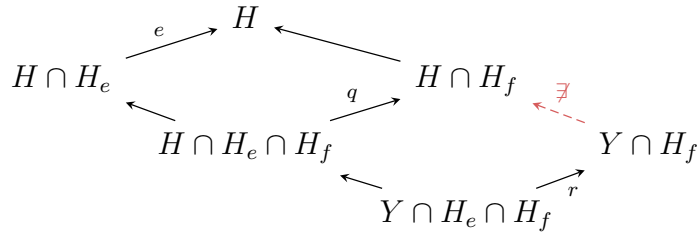
\[\begin{tikzcd}[column sep = tiny, row sep = tiny]
	& H \\
	{H \cap H_e} && {H \cap H_f} \\
	& {H \cap H_e \cap H_f} && {Y \cap H_f} \\
	&& {Y \cap H_e \cap H_f}
	\arrow["e", from=2-1, to=1-2]
	\arrow[from=2-3, to=1-2]
	\arrow[from=3-2, to=2-1]
	\arrow["q", from=3-2, to=2-3]
	\arrow["{\not \exists}"'{pos=0.7}, color={rgb,255:red,214;green,92;blue,92}, dashed, from=3-4, to=2-3]
	\arrow[from=4-3, to=3-2]
	\arrow["r"', from=4-3, to=3-4]
\end{tikzcd}\]

We show now that \(q\) and \(e\) are distinct: $H \neq H_f \cap H$. 
If $H = H \cap H_f$, then $H \subseteq H_f$.
We assumed that $Y \subseteq H$, so $Y = Y \cap H_f = Y$ and $Y \cap H_e \cap H_f = Y \cap H_e$, so $r=r_1$, a contradiction.

If $r \prec q$ then we are done as \(r\prec^F q\leq e\). 

If $r \not\prec q$, then we have constructed $r < q < e$ such that $q \in \mc{\O}^c$, through a compatibility failure.
Hence, we can repeat our construction of $q$ given $r<^F e$ on the pair $r<^F q$ to find $q'$ such that $r<^F q'<q$, and so on.
We will eventually obtain a transfer $s$ such that $r \prec s < e$ and $r \prec^F s$, which implies that $s \notin \mc{\O}'$.
There must exist a chain of covering relations $s \prec s_1 \prec \ldots \prec e$, and since $s \in \mc{\O}'^c$ then $e \in \mc{\O}'^c$ as desired.
\end{proof}

\begin{remark}
Note that \cref{thm:maximal compatible disk-like via cover relations} does not in general apply to transfer systems that are not disklike. Consider the following two $C_{p^2q}$-transfer systems. Let the transfer system on the right be $\O$. If we use the formula in \cref{thm:maximal compatible disk-like via cover relations} to try to find 
the maximal compatible transfer system,
we obtain the transfer system on the left. However, this is not $\mc{\O}$ as there is a compatibility failure highlighted by the red dashed arrow on $\O$.
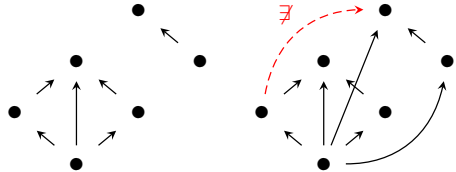
\[\begin{tikzcd}[column sep = tiny, row sep = tiny]
	&& \bullet &&&& \bullet \\
	& \bullet && \bullet && \bullet && \bullet \\
	\bullet && \bullet && \bullet && \bullet \\
	& \bullet &&&& \bullet
	\arrow[from=2-4, to=1-3]
	\arrow[from=2-8, to=1-7]
	\arrow[from=3-1, to=2-2]
	\arrow[from=3-3, to=2-2]
	\arrow["{\not\exists}", color=red, curve={height=-18pt}, dashed, from=3-5, to=1-7]
	\arrow[from=3-5, to=2-6]
	\arrow[from=3-7, to=2-6]
	\arrow[from=4-2, to=2-2]
	\arrow[from=4-2, to=3-1]
	\arrow[from=4-2, to=3-3]
	\arrow[from=4-6, to=1-7]
	\arrow[from=4-6, to=2-6]
	\arrow[curve={height=18pt}, from=4-6, to=2-8]
	\arrow[from=4-6, to=3-5]
	\arrow[from=4-6, to=3-7]
\end{tikzcd}\]
\end{remark}

\begin{example}\label{ex:algorithm example} 
Consider the following disklike $C_{p^2q^2}$-transfer system, $\O$, with $\mc{\O}$ indicated in bold blue edges.
Using this example, we illustrate all possible logical reasons for the presence or absence of an edge in a generic $\mc{\O}$ using \cref{thm:maximal compatible disk-like via cover relations}:
$e_1,f_1$ and $g_1$ are all in $\mc{\O}$ as they restrict into nothing else; $e_2$ is in $\mc{\O}^c$ as $e_1 \coverFailure e_2$; $e_3$ is in  $\mc{\O}^c$ as $e_2 \cover e_3$ and $e_2 \in \mc{\O}^c$; and $f_2 \in \mc{\O}$ as it covers elements of $\mc{\O}$, $f_1$ and $g_1$, through compatibility successes.

\begin{center}
\begin{tikzcd}[ampersand replacement = \&, nodes={scale=0.8},column sep = small, row sep = small]
\& \& \bullet\\
\& \bullet \& \& \bullet\\
\bullet \&\& \bullet \& \& \bullet\\
\& \bullet \&\& \bullet\\
\& \& \bullet

\arrow[from=2-2, to=1-3, "e_3", outer sep=-.5]
\arrow[from=3-1, to=1-3, bend left=37]
\arrow[from=4-4, to=1-3, bend right=10]
\arrow[from=3-3, to=2-4, "e_2", outer sep=-.5]
\arrow[from=4-4, to=2-2, bend left]
\arrow[from=4-2, to=2-4, bend right]
\arrow[from=5-3, to=2-2, bend left=20]
\arrow[from=5-3, to=1-3, bend left]
\arrow[from=5-3, to=2-4, bend right=18]
\arrow[from=5-3, to=3-1, bend left=37]
\arrow[from=4-4, to=2-4, bend right=20]

\arrow[from=3-1, to=2-2, very thick, crossing over, cblue]
\arrow[from=4-4, to=3-5, "e_1"', very thick, crossing over, outer sep=-.5, cblue]
\arrow[from=4-2, to=3-3, very thick, crossing over, cblue]
\arrow[from=5-3, to=3-5, bend right=37, very thick, cblue]
\arrow[from=5-3, to=4-2, "f_1", very thick, outer sep=-1.5, cblue]
\arrow[from=5-3, to=4-4, "g_1"', very thick, outer sep=-.5, cblue]
\arrow[from=5-3, to=3-3, "f_2"', very thick, crossing over, pos=.35, cblue]
\arrow[from=4-4, to=3-3, very thick, crossing over, cblue]
\end{tikzcd}
\end{center}

\end{example}

\subsection{Disklike Consequences and Examples}

\cref{thm:maximal compatible disk-like via cover relations} immediately yields an algorithm for computing the maximal compatible transfer system.
We observe that the initialization $\mc{\O}\coloneqq  \min \O$ can be done as transfers in $\O$ which only restrict to trivial transfers are trivially in $\mc{\O}$ (\cref{cor:an edge with no restrictions is in the maximal compatible}), and the reason we may work with conjugacy classes of transfers, $[m]$, is because $\mc{\O}$ is closed under conjugation.

\alglanguage{pseudocode}
\begin{algorithm}[H]
\renewcommand{\thealgorithm}{\thetheorem}
\caption{Given a disklike $G$-transfer system $\O$ computes $\mc{\O}$.}
\begin{algorithmic}\label{alg:maximal compatible of disk-like}
\State Initialize $\mc{\O} \coloneqq  \min \O$ under the $\leq$ order, and $Q \coloneqq  \O \setminus \mc{\O}$.
\While{$Q$ is non-empty}
\State Let $m$ be a minimal element of $Q$ under the $\leq$ order, and delete $[m]$ from $Q$.
\If{$\forall r\in \O$ such that $r \prec m$, we have that $r \in \M(\O)$ and $r \coverSuccess m$}
\State Add $[m]$ to $\mc{\O}$.
\EndIf
\EndWhile
\State \Return $\mc{\O}$.
\end{algorithmic}
\end{algorithm}

We now analyze the performance of \cref{alg:maximal compatible of disk-like}.
First, we bound the number of non-reflexive transfers in a disk-like transfer system $\O$, denoted \defn{$|\O|$}, by the number of cover relations in the restriction poset of $\O$, denoted \defn{$C_{\O}$}.
Note, most disklike transfer systems have many more cover relations than non-reflexive transfers.

\begin{lemma}\label{lem:number of edges of a disk-like transfer sytem are bounded by number of cover relations}
Every disklike $G$-transfer system, $\O$, satisfies $|\O| \leq 2C_{\O} + 1$.
\end{lemma}

\begin{proof}
Let $D$ be the maximal set of disklike generators of $\O$.
Then we have the partition $\O = D\sqcup (\O \setminus D)$.
We observe that every transfer in $e\in \O\setminus D$ has a disklike generator, and hence has at least one cover relation above it $e\prec q$.
Thus there is a surjection from cover relations to $\O\setminus D$, given by $(e \prec q )\mapsto e$,
implying $|\O\setminus D| \leq C_{\O}$.
We shall show that $|D|\leq |\O \setminus D|+1$, and thus conclude that $|\O| \leq 2 |\O\setminus D| + 1 \leq 2 C_{\O} + 1$.

Suppose that $\O$ has a single disklike generator, then $|D|=1 \leq 1$.
If $\O$ has more than one disklike generator, then it has a unique minimal disklike generator $N_{\mathcal O}\to G$ where $N_{\mathcal O}$ is meet of all subgroups $H$ such that $H \to G \in D$, (see \cref{lem:no universal transfer means image of unit}).
By restriction, there are transfers $N_{\mathcal O}\to H$ for every every other disklike generator $H\to G \in D$.
Thus we have an injection from the set of non-minimal disk like generators to $\mathcal{O}\setminus D$, given by $H\to G \mapsto N_{\mathcal O}\to H$.
This implies $|D|-1 \leq |\O\setminus D|$ as desired.
\end{proof}

We now analyze the worst case asymptotic performance of \cref{alg:maximal compatible of disk-like}.
Recall that if $f,g:A\to \mathbb{N}$, then in big O notation we say $f =O(g)$ if there exists a constant $K \in \mathbb{N}$ such that $\forall a \in A, f(a)\leq K g(a)$.
Thus an algorithm is said to have worst case time complexity $O(g)$ if a function $f$ measuring the number of operations of the algorithm satisfies $f=O(g)$.

\begin{proposition}\label{prop:disk-like maximal compatible alg time complexity}
\cref{alg:maximal compatible of disk-like} has a $O(C_{\O})$ worst case time complexity across all transfer systems across all groups, where $C_{\O}$ is the number of cover relations in the restriction poset of $\O$.
Moreover, \cref{alg:maximal compatible of disk-like} is optimal among all algorithms which must compute the restriction poset of $\O$
\end{proposition}

\begin{proof}
We shall show that the initialization time and the main while loop of the algorithm both have a worst case time complexity of $O(C_{\O})$.
We note that in the worst case, our group is abelian, and thus we cannot improve performance via clever use of conjugacy classes.

The initialization cost is $O(C_{\O})$.
To initialize our queue and to determine minimal transfers in $\O$ under the restriction order $\leq$, we just need to topologically sort the transfers of $\O$ under $\leq$, i.e. pick a total ordering of the transfers of $\O$ which is compatible with $\leq$.
Optimal algorithms for topologically sorting a graph with $|V|$ vertices and $|E|$ edges have a worst case time complexity of $O(|V|+|E|)$ (see for example \cite{cormen2022introduction}), thus our topological sort has a time complexity of $O(|\O|+ C_{\O} )$. 
However, \cref{lem:number of edges of a disk-like transfer sytem are bounded by number of cover relations} says that in a disklike transfer system $|\O| \leq 2 C_{\O}+ 1$, and thus $O(|\O|+ C_{\O}) = O(C_{\O})$ (i.e. the asymptotic growth of the sum of two things depends solely on the larger).
The main while loop of the algorithm also has a worst case time complexity of $O(C_{\O})$, as in this particular case the algorithm inspects every single cover relation, yielding the claimed total complexity $O(C_{\O})$.

Finally, we can see that this worst case time complexity is optimal amongst all algorithms which must construct the graph of the restriction poset, as this graph has $|\O|$ vertices, and $C_{\O}$ edges, and thus takes $O(|\O|+ C_{\O}) = O(C_{\O})$ time to construct.
\end{proof}

\begin{remark}\label{rem:optimality of alg}
Hypothetically, the best possible algorithm for computing the maximal compatible transfer system would have a time complexity of $O(|\O|)$, as one must decide whether or not each transfer is in $\mc{\O}$ or not.
Given that $|\O| \leq 2 C_{\O}+1$, it is clear that \cref{alg:maximal compatible of disk-like} will not be optimal for certain subsets of groups, and certain subsets of disklike transfer systems.
For example, one can show that the maximal compatible transfer system  of a $C_{p^n}$-transfer system, $\O$, can be computed in $O(\min (|\O|, n))$ time. 
However, it seems likely that \cref{alg:maximal compatible of disk-like} is asymptotically optimal among all algorithms for computing the maximal compatible transfer system of arbitrary disklike $G$-transfer system, i.e., every general algorithm must check some linear proportion of cover relations.
\end{remark}

\begin{remark}
When working with \cref{alg:maximal compatible of disk-like} for $G$-transfer systems, $G$ non-abelian, one could hope to replace $\sub(G)$ with its quotient by conjugacy action, $\sub(G)/G$.
Unfortunately, categorical transfer systems on $\Sub(G)/G$ do not always lift to $G$-transfer systems \cite{balchin2022lifting}.
In particular, there exist disklike $\Sigma_4$-transfer systems which do not lift, and for which the maximal compatible transfer system cannot be computed via lifting the maximal compatible transfer system of its quotient.
\end{remark}

\subsection{Disklike Conjectured Characterization}

Consider the following conjecture.

\begin{conjecture}\label{conj:maximal compatible disk-like via single restriction}
    Let $\O$ be a disklike $G$-transfer system, then the maximal compatible transfer system can be written as the following set:
    \begin{align}
        \mc{\O} = \{ e \in \O \, | \, \forall r < e, r <^S e \},
    \end{align}
    where every restriction is in $\O$. 
    Its complement, $\mc{\O}^c \coloneqq \O\backslash \mc{\O}$, can be written as:
    \begin{align}
        \mc{\O}^c = \{ e \in \O \, | \, \exists r< e, r <^F e \}.
    \end{align}
\end{conjecture}

The disklike condition is important here. 
For example, in \cref{ex:non disk-like counter example to good conjecture} the transfer system is not disklike and this conjecture does not find the correct $\mc{\O}$ as the highest transfer is in the complement but does not directly admit a compatibility failure.
The reader may verify that every example of a disklike $G$-transfer system given in this paper satisfies this conjecture, see for instance \cref{ex:hull preserves compatibility,ex:algorithm example}.
\cite[Proposition 3.2 and Corollary 2.20]{Hill-Meng-Nan} makes it straightforward to prove the result holds for $C_{p^n}$, $p$ prime.
Additionally, it is an immediate consequence of \cref{cor:universal-transfer}, that the conjecture holds if, and only if, it holds for every $\O$ with $1 \to G \in \O$. 

Further, we can consider a minimal set of generators of a transfer system $\O$, say $B$, such that $T(B) = \O$ but if we remove any edge from $B$ it will not generate $\O$. Note that a minimal set of generators will never have a reflexive edge.

\begin{definition}\label{def:complexity}
The \defn{complexity of a $G$-transfer system} $\O$ is defined to be the number of edges in a minimal set of generators of $\O$.
\end{definition}

\begin{definition}[{\cite[Definition 4.1]{adamyk2025minimal}}]
The \emph{complexity of a group $G$} is the maximal complexity of all $G$-transfer systems.
\end{definition}

We experimentally verified the conjecture (using \cite{GAP}) under the following conditions.

\begin{proposition}\label{prop:experimental verification of conj}
\cref{conj:maximal compatible disk-like via single restriction} holds for all disklike $G$-transfer systems $\O$ such that,
\begin{enumerate}[\emph{(\arabic*)}]
    \item $Order(G)\leq 15$, and $e\to G \in \O$, 
    \item $Order(G)\leq 63$, $e\to G \in \O$, and $Complexity(\O)\leq 2$, or
    \item $G=\Sigma_n$ for $n\leq 6$, and $Complexity(\O)\leq 2$. 
\end{enumerate}
\end{proposition}

\begin{remark}
The most general case we were able to prove for the conjecture is: for every abelian group $G$, every disklike transfer system of complexity $\leq 2$ with the universal transfer $1 \to G$ satisfies \cref{conj:maximal compatible disk-like via single restriction}.
We were unable to extend this result to transfer systems of arbitrary complexity, but one plausible bridging lemma might be: If $\O$ is a (categorical) disklike transfer system for which every transfer system contained in $\O$ of complexity $\leq 2$ satisfies \cref{conj:maximal compatible disk-like via single restriction}, then $\O$ satisfies \cref{conj:maximal compatible disk-like via single restriction}.
\end{remark}

The following example demonstrates that there exist disklike categorical transfer systems which do not satisfy \cref{conj:maximal compatible disk-like via single restriction}.
This implies the conjecture may only hold for disklike $G$-transfer systems, and that there is no purely diagrammatic proof as there was previously.
The reader may verify that the proofs of \cref{prop:computing Om via restricting edges} and \cref{thm:maximal compatible disk-like via cover relations} also apply to categorical transfer systems, which lets us apply them here.

\begin{example}\label{ex:categorical disk-like transfer system is not good}
Consider the following disklike categorical transfer system $\O$ on a bounded lattice, say $P$, with Hasse diagram indicated by dotted edges. 
Using \cref{thm:maximal compatible disk-like via cover relations} we can compute that the only non-reflexive transfer in $\mc{\O}$ is $r'$.
Observe that $\mc{\O}$ does not satisfy \cref{conj:maximal compatible disk-like via single restriction} as $r'<^Fr<e$ implies $e\in \mc{\O}^c$, but there is no edge $q$ such that $q<^Fe$.
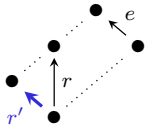
\[\begin{tikzcd}[column sep = tiny, row sep = tiny,nodes={inner sep=0.05cm},]
	&& \bullet \\
	& \bullet && \bullet \\
	\bullet \\
	& \bullet
	\arrow[dotted, no head, from=2-2, to=1-3]
	\arrow["e"', from=2-4, to=1-3]
	\arrow[dotted, no head, from=3-1, to=2-2]
	\arrow["r"', from=4-2, to=2-2]
	\arrow[dotted, no head, from=4-2, to=2-4]
	\arrow["{r'}", from=4-2, to=3-1, very thick, cblue]
\end{tikzcd}\]
\end{example}

\section{Maximal Compatibility is Preserved by Inflation}
\label{sec:inflation}

In this section, we fix a normal subgroup \(N\) and associated quotient map \(p\colon G\to G/N\). The set-theoretic inverse canonically identifies the subgroup lattice of \(G/N\) with those subgroups of \(G\) which contain \(N\), and we can use this to pull back any transfer system on \(G/N\).

\subsection{Inflation Formulae}\label{ssec:inflation characterization}

We first recall the definition of $p^*$, denoted $p_L^{-1}$ in \cite{RubinOperadicLifts}.

\begin{definition}[{\cite[Section 5.1]{RubinOperadicLifts}}]\label{def:inflation}
Let
\[
p^{-1}(\O)=\{p^{-1}K\to p^{-1}H \,\mid \, K\to H \text{ in }\O\}.
\]
The \defn{inflation} of \(\O\), $p^*\O$, is the transfer system generated by \(p^{-1}\cO\):
\[
    p^*\O \coloneqq  T( p^{-1}\O).
\]
\end{definition}

Rubin simplifies this definition. 

\begin{lemma}[{\cite[Lemma 5.7]{RubinOperadicLifts}}]
\label{lem:pre-image closed under conjugation}
For any \(G/N\)-transfer system \(\O\), we have 
\[
p^*\O = \Comp(\Res(p^{-1}\O)).
\]
\end{lemma}

We can further simplify this producing an almost disklike condition.
This lemma was also shown independently in \cite[Proposition 1.20]{aitken2026equivariant}.

\begin{lemma}\label{lem:in inflation iff restriction}
For any \(G/N\)-transfer system \(\O\), we have 
\[
p^*\O = \Res(p^{-1}\O).
\]
Moreover, a transfer $K\to H \in p^*\O$, if and only if $KN \to HN \in p^{-1}\O$ and $K=KN\cap H$.
\end{lemma}

\begin{proof}
We prove the second part first.
By the definition of restriction, if $KN \to HN \in p^{-1}\O$ and $KN\cap H=K$, then $K \to H \in \Res(p^{-1}\O)$.
Suppose $K\to H$ is in $\Res(p^{-1}\O)$.
This implies there is some $PN\to QN \in p^{-1}\O$ with $H\leq QN$ and $K=PN\cap H$. 
Then, necessarily, $KN\leq PN$ and $HN\leq QN$, giving us the following picture:

\[
\begin{tikzcd}[column sep = small, row sep = small]
    &           & QN &            \\
    & HN        &    & PN \ar[ul] \\
  H &           & KN &            \\
    & K \ar[ul] &    &
\end{tikzcd} 
\]
Observe that $KN\leq PN\cap HN$, as both $N$ and $K$ are subgroups of $PN\cap HN$.
For the converse containment, let $p\in P$, $h\in H$ and $n_1,n_2\in N$ be such that $pn_1=hn_2$ (i.e. a generic element of $PN\cap HN$). 
Then $h=pn_1n_2^{-1}\in PN\cap H$, which is equal to $K$ by assumption. 
Hence, $hn_2\in KN$.
As $hn_2$ was chosen to be arbitrary within $PN \cap HN$, this yields $PN\cap HN \leq KN$, giving the desired equality $KN = PN\cap HN$, which shows that \(KN\to HN\) is in \(p^{-1}\cO\).

Finally, we use this to show that $\Res(p^{-1}\O) = \Comp(\Res(p^{-1}\O))$.
Suppose $K\to H \in \Comp(\Res(p^{-1}\O))$ admits a decomposition as a composite of $K\to X_1, X_1 \to X_2, \ldots, X_n \to H$ in $\Res(p^{-1}\O)$.
Given the preceding characterization of $\Res(p^{-1}\O)$, we have that $KN\to X_1N, X_1N \to X_2N, \ldots, X_nN \to HN \in p^{-1}\O$, yielding the composite $KN \to HN \in p^{-1}\O$.
All that remains is to show $K=KN \cap H$.
Given $K = KN \cap X_1, X_1 = X_1N \cap X_2, \ldots,$ and $X_n = X_nN \cap H$, we observe that $K = KN \cap X_1 =  KN \cap (X_1N \cap X_2) = \ldots = KN \cap (X_1N \cap X_2N \cap \ldots \cap X_nN \cap H)$.
However, as $KN\leq X_1N \leq \ldots \leq X_nN$, this implies $K=KN \cap H$ as required.
\end{proof}

One critical consequence is: 
the portion of $p^*\O$ lying within the interval from \(N\) to \(G\) is exactly $p^{-1}\O$.

\begin{corollary}
\label{cor:p^* in the interval}
For $N \leq H\leq K \leq G$, the transfer $H\to K$ is in $p^*\O$ if and only if $H\to K$ is in $p^{-1}\O$ if and only if $H/N \to K/N $ is in $\O$.
\end{corollary}

Here is an illustrative example of these results.

\begin{example}\label{ex:inflation example}
Let $q$ and $r$ be distinct primes, and let $p:C_{{q}^3{r}^3}\onto C_{{q}^2{r}^2}$ be the projection with kernel $C_{{q}{r}}$.
In the below diagram we display a disklike $C_{{q}^2{r}^2}$-transfer system $\O$ on the left, and the bold blue edges make up its maximal compatible transfer system $\mc{\O}$, which was computed in \cref{ex:algorithm example}.
On the right is the inflation $p^*\O$, and $p^*\mc{\O}$ is made up by the bold blue edges, computed using \cref{thm:maximal compatible disk-like via cover relations}.
The green dashed line borders the copy of $\O$ as the interval $[C_{{q}{r}}, C_{{q}^3{r}^3}]$.

\begin{center}
\vspace{0.0cm}
\begin{tikzcd}[ampersand replacement = \&,nodes={inner sep=0.1cm}, column sep = small, row sep = small]
\&\&\& \bullet \\
\&\&\bullet \&\& \bullet\\
\&\bullet \&\& \bullet\&\& \bullet\&\&\\ 
\phantom{\bullet} \&\&\bullet \&\& \bullet\&\& \phantom{\bullet}\\ 
\&\phantom{\bullet} \&\& \bullet\&\& \phantom{\bullet}\&\&\\ 
\&\&\phantom{\bullet} \&\& \phantom{\bullet}\\
\&\&\& \phantom{\bullet}
\arrow[from=2-3, to=1-4, outer sep=-.5]
\arrow[from=3-2, to=1-4, curve={height=-12pt}]
\arrow[from=4-5, to=1-4, bend right=10]
\arrow[from=3-4, to=2-5, outer sep=-.5]
\arrow[from=4-5, to=2-3, bend left]
\arrow[from=4-3, to=2-5, bend right]
\arrow[from=5-4, to=2-3, bend left=20]
\arrow[from=5-4, to=1-4, bend left]
\arrow[from=5-4, to=2-5, bend right=18]
\arrow[from=5-4, to=3-2, curve={height=-12pt}]
\arrow[from=4-5, to=2-5, bend right=20]
\arrow[from=3-2, to=2-3, very thick, crossing over, cblue]
\arrow[from=4-5, to=3-6, very thick, crossing over, cblue]
\arrow[from=4-3, to=3-4, very thick, crossing over, cblue]
\arrow[from=5-4, to=3-6, curve={height=12pt}, very thick, crossing over, cblue]
\arrow[from=5-4, to=4-3, very thick, outer sep=-1.5, cblue]
\arrow[from=5-4, to=4-5, very thick, outer sep=-.5, cblue]
\arrow[from=5-4, to=3-4, very thick, crossing over, pos=.35, cblue]
\arrow[from=4-5, to=3-4, very thick, crossing over, cblue]
\end{tikzcd}
\raisebox{0.5cm}{$\xrightarrow{p^*}\quad\quad$}
\begin{tikzcd}[ampersand replacement = \&,nodes={inner sep=0.05cm}, column sep = small, row sep = small,
every matrix/.append style = {name=m},
remember picture]
\&\&\& \bullet \\
\&\&\bullet \&\& \bullet\\
\&\bullet \&\& \bullet\&\& \bullet\&\&\\ 
{\bullet} \&\&\bullet \&\& \bullet\&\& {\bullet}\\ 
\&{\bullet} \&\& \bullet\&\& {\bullet}\&\&\\ 
\&\&{\bullet} \&\& {\bullet}\\
\&\&\& {\bullet}
\arrow[from=2-3, to=1-4, outer sep=-.5]
\arrow[from=3-2, to=1-4, curve={height=-12pt}]
\arrow[from=4-5, to=1-4, bend right=10]
\arrow[from=3-4, to=2-5, outer sep=-.5]
\arrow[from=4-5, to=2-3, bend left]
\arrow[from=4-3, to=2-5, bend right]
\arrow[from=5-4, to=2-3, bend left=20]
\arrow[from=5-4, to=1-4, bend left]
\arrow[from=5-4, to=2-5, bend right=18]
\arrow[from=5-4, to=3-2, curve={height=-12pt}]
\arrow[from=4-5, to=2-5, bend right=20]
\arrow[from=3-2, to=2-3, very thick, crossing over, cblue]
\arrow[from=4-5, to=3-6, very thick, crossing over, cblue]
\arrow[from=4-3, to=3-4, very thick, crossing over, cblue]
\arrow[from=5-4, to=3-6, curve={height=12pt}, very thick, crossing over, cblue]
\arrow[from=5-4, to=4-3, very thick, outer sep=-1.5, cblue]
\arrow[from=5-4, to=4-5, very thick, outer sep=-.5, cblue]
\arrow[from=5-4, to=3-4, very thick, crossing over, pos=.35, cblue]
\arrow[from=4-5, to=3-4, very thick, crossing over, cblue]
\arrow[from=5-6, to=4-7, very thick, cblue]
\arrow[curve={height=-12pt}, from=6-3, to=4-1]
\arrow[from=6-3, to=5-2, very thick, cblue]
\arrow[curve={height=12pt}, from=6-5, to=4-7, very thick, cblue]
\arrow[from=6-5, to=5-6, very thick, cblue]
\end{tikzcd}
\begin{tikzpicture}[
remember picture, overlay]
\myroundpoly[ao, very thick, dashed]{m-1-4,m-3-6,m-5-4,m-3-2}{0.5cm};
\end{tikzpicture}
\end{center}
Applying \cref{cor:p^* in the interval,lem:in inflation iff restriction} in turn, we observe that the pre-image has embedded isomorphic copies of $\O$ and $\mc{\O}$ into the interval $[C_{{q}{r}}, C_{{q}^3{r}^3}]$, and that every transfer outside of this interval, for both $p^*\O$ and $p^*\mc{\O}$, is generated by restriction.

\end{example}

\subsection{Preservation}\label{sec:Preservation}

Having characterized inflation, we can now prove \cref{thm:inflation}. 
We start by proving that $p^*$ preserves both disklike and saturated transfer systems.

\begin{proposition}[\cref{thm:inflation}(1)]
\label{prop: inflation preserves disklike}
If $\O$ is a disklike transfer system on \(G/N\), then $p^*\O$ is also disklike.
\end{proposition}
\begin{proof}

Since $\O$ is disklike and by definition of $p^{-1}$, then every transfer in $\O$ and $p^{-1}\O$ is a restriction of a transfer with target $G$.
Recall that $\Res(p^{-1}\O) = p^*\O$, therefore every transfer of $p^*\O$ is a restriction of an transfer with target $G$ as desired. \end{proof}

\begin{proposition}[\cref{thm:inflation}(2)]\label{prop: inflation preserves saturation}
If $\O$ is saturated, then so is $p^*\O$.
\end{proposition}

\begin{proof}
Let $K \leq J \leq H \leq G$.
Suppose $\O$ is saturated, and let $K\to J$, $K\to H$ be transfers in $p^*\O$, yielding the left diagram of \cref{eqn:inflated saturation failure?}.
Applying \cref{lem:in inflation iff restriction}, we can consider the analogous diagram in $p^{-1}\O$. 
Given $\O$ is saturated, the transfer $JN\to HN$ must be in $p^{-1}\O$.
\begin{equation}\label{eqn:inflated saturation failure?}
\begin{tikzcd}
K & J & H
\arrow[from=1-1, to=1-2, "p^*\O"']
\arrow[curve={height=-12pt}, from=1-1, to=1-3, "p^*\O"]
\arrow[dashed, from=1-2, to=1-3]
\end{tikzcd}
\quad
\xrightarrow{-\lor N}
\quad
\begin{tikzcd}
	KN & JN & HN
	\arrow[from=1-1, to=1-2, "p^{-1}\O"']
	\arrow[curve={height=-12pt}, from=1-1, to=1-3, "p^{-1}\O"]
	\arrow[dashed, from=1-2, to=1-3, "p^{-1}\O"']
\end{tikzcd}
\end{equation}

It remains to show the transfer $JN\to HN$ restricts to the transfer $J\to H$, i.e. $JN\cap H=J$.
Since $J \leq H$, $J \leq JN \cap H$.
Let $h \in JN \cap H$, i.e. $h=jn$ for some $j\in J$, $n\in N$.
Then we have $n=j^{-1}h\in H$ and by assumption $J\leq H$, which together gives $n\in H\cap N\leq H\cap KN$.
Since $K \to H \in p^*\O$, \cref{lem:in inflation iff restriction} shows that $H\cap KN=K$.
Hence $n\in K$ and $h=jn\in J$.
Therefore $J = JN \cap H$ as desired. \end{proof}

Inflation also preserves inclusions.
\begin{lemma}\label{lem:inflation monotone}
Suppose $\O_1, \O_2$ are $G/N$ transfer systems such that $\O_1 \subseteq \O_2$. 
Then $p^*\O_1\subseteq p^*\O_2$.
\end{lemma}
\begin{proof}
This follows from the pre-image, and $\Res$ each satisfying the same condition.
\end{proof}

To prove the remaining preservation results, 
we will need the following two group theoretic facts which are well known but we provide proofs for convenience.
\begin{lemma}\label{lem: intersection of join implies distributive N}
Let $A,B \leq G$ and $N \trianglelefteq G$.
If $AB\cap AN=A$, then $(A\cap B)N=AN\cap BN$.
\end{lemma}

\begin{proof}
It is immediate that $(A\cap B)N \leq AN\cap BN$.
Suppose $a\in A$, $b\in B$, and $n_1,n_2\in N$ such that $bn_1=an_2$. 
Then, we have that $b=an_2n_1^{-1}$.
Hence, $b\in AN$, and in particular, $b\in AB\cap AN=A$.
Thus, we have shown $bn_1\in (A\cap B)N$.
Therefore $(A \cap B) N = AN \cap BN$. \end{proof}
\begin{lemma}\label{lem:group product containment}
Let $A,B \leq G$ and $N \trianglelefteq G$.
If $N\leq A \leq BN$, then $A=(A\cap B)N$.
\end{lemma}

\begin{proof}
    Let $A,B$, and $N$ be as above such that $N \leq A \leq BN$.
    Since $(A\cap B),N \leq A$, we have that $(A\cap B)N \leq A$.
    Consider $a\in A$, since $A\leq BN$ and $N$ is normal, we have that there is some $b\in B$, $n\in N$ such that $a=bn$.
    Multiplying by \(n^{-1}\) $n$, we have that $b=an^{-1}$.
    As $N\leq A$, then $b = an^{-1} \in A$ and hence $bn \in (A\cap B)N$.
    Therefore $A = (A \cap B)N$. \end{proof}

\begin{proposition}\label{prop: inflation preserves saturated compatibility}
Let $\Oa,\Om$ be $G/N$-transfer systems such that $\Om$ is saturated. If $(\Oa,\Om)$ is a compatible pair, then $(p^*\Oa,p^*\O_m)$ is a compatible pair.
\end{proposition}

\begin{proof}
Assume $\Oa$ and $\Om$ are as above such that $(\Oa,\Om)$ is a compatible pair.
To check if $(p^*\Oa,p^*\O_m)$ is a compatible pair, 
we need to show $p^*\O_m\subseteq p^*\Oa$ and that there are no compatibility failures.
The former is immediate from \cref{lem:inflation monotone}.

Checking for compatibility failures amounts to checking that, for the left diagram of \cref{eq:compatibility failure?}, if $K \to H \in p^*\Om$ and $K \cap J \to K \in p^*\Oa$ then $J\to H$ must be in $p^*\Oa$. 
By \cref{lem:in inflation iff restriction} since $K \to H \in p^*\Om$ then $KN\cap H = K$. We can consider the analogous diagram on the interval $[N,G]$ (right).
\begin{equation}\label{eq:compatibility failure?}
\begin{tikzcd}[column sep = tiny, row sep = tiny]
&H&\\
K\ar[ur,"p^*\O_m"]&&J\arrow[ul,dashed]\\
&K\cap J\arrow[ur,"p^*\O_m",swap]\ar[ul,"p^*\Oa"]&
\end{tikzcd}
\quad
\xrightarrow{-\lor N}
\quad
\begin{tikzcd}[column sep = tiny, row sep = tiny]
&HN&\\
KN\ar[ur,"p^{-1}\O_m"]&& JN\arrow[ul,dashed]\\
&(K\cap J)N\arrow[ur,"p^{-1}\O_m",swap]\ar[ul,"p^{-1}\Oa"]&
\end{tikzcd}
\end{equation}

We will use the compatibility of $(\Oa,\O_m)$ and \cref{cor:p^* in the interval} to establish that $JN\to HN \in p^*\Oa$ . 
As is, the above diagram on the right is not a compatibility diagram as we need to prove that the bottom subgroup $(K \cap J)N$ is equivalent to $KN \cap JN$, then we will know that $JN \to HN \in p^{-1}\Oa$ by compatibility.
The containment $(K\cap J)N\leq KN\cap JN$ follows as $\cap$ is a meet of a lattice.
To show the reverse containment, let $k\in K$, $\ell\in J$, and $n_1,n_2\in N$ such that $kn_1=\ell n_2 \in KN \cap JN$.
 Then, multiplying by $n_2^{-1}$, we see $\ell=kn_1n_2^{-1} \in KN \cap J$.
As $J\leq H$, we have that $KN \cap J\leq KN \cap H=K$.
Hence, $\ell \in K \cap J$, and $\ell n_2\in (K\cap J)N$.
Therefore $(K\cap J)N=KN\cap JN$.
The compatibility of $(\Oa,\O_m)$ then implies the existence of the transfer $JN \to HN$.

We now show that the transfer $JN \to HN$ restricts to $J\to H$. 
Because $KJ \leq H$, the transfer $K \to H$ restricts to $K\to KJ$.
Applying \Cref{lem:in inflation iff restriction} we get that $KN \to KJN \in p^{-1}(\O)$, then repeating the argument from above, we get that $JN\to KJN\in p^{-1}\Oa$ by compatibility.
We will shortly prove that $JN \to KJN$ restricts into $J\to KJ \in p^*\Oa$, providing the diagram below:
$$
\begin{tikzcd}[column sep = small, row sep = small]
&H&\\
&KJ\ar[u,dashed]&\\
K\ar[uur,"p^*\O_m"] \ar[ur,"p^*\O_m",swap, near start] &&J\ar[ul,"p^*\Oa", near start] \ar[uul,dashed]\\
&K \cap J. \ar[ul,"p^*\Oa"]\ar[ur,"p^*\O_m",swap]&
\end{tikzcd}
$$
Given $\Om$ is assumed to be saturated, and inflation preserves saturation  (\cref{prop: inflation preserves saturation}), $p^*\O_m$ is saturated as well.
Thus, since $K \to H \in p^*\O_m$ then $KJ\to H \in p^*\O_m$.
We can then compose $KJ\to H$ with the claimed transfer $J\to KJ$ to obtain the desired transfer $J\to H \in p^*\Oa$.

To complete the proof, 
it is sufficient to show that $JN \to KJN$ restricts into $J\to KJ$, or equivalently that $KJ \cap JN = J$.
By inspection, $J\leq KJ \cap JN$.
To show that $KJ \cap JN \leq J$, let $m=\ell n\in KJ \cap JN$ where $\ell \in J$ and $n\in N$.
Multiplying by $\ell^{-1}$, we see that $n=\ell^{-1}m\in KJ\cap N$.
Noting that $KJ\cap N\leq KJ\cap KN$ and applying \cref{lem:in inflation iff restriction} to the transfer $K\to KJ$, we have that $KJ\cap KN=K$ and hence $n\in K\cap N$.
Mirroring the previous step, we note that $K\cap N\leq K\cap (K\cap J)N$.
Applying \cref{lem:in inflation iff restriction} once more, this time to the transfer $K\cap J\to K$, we yield that $K\cap (K\cap J)N=K\cap J$, and hence $n\in K\cap J$.
This establishes that $m=\ell n \in J$, and thus $KJ\cap JN = J$
as desired. \end{proof}

\begin{proposition}[\cref{thm:inflation}(4)]
\label{prop: inflation preserves max compat}
If $(\O,\O_m)$ is a maximal compatible pair, then so is $(p^*\O,p^*\O_m)$.
\end{proposition}

\begin{proof}

We seek to show that $\mc{p^*\O} = p^*{\mc{\O}}$.
We first consider the direction $p^*{\mc{\O}} \subseteq \mc{p^*\O}$.
Since $\mc{\O}$ is saturated and compatible with $\O$, we may apply \cref{prop: inflation preserves saturated compatibility} to deduce that $p^*{\mc{\O}}$ is compatible with $p^*{\O}$.
This implies by \cref{prop:Maximal compatible characterises all compatibility} that $p^*{\mc{\O}} \subseteq \mc{p^*\O}$.

To show the other direction, $\mc{p^*\O} \subseteq p^*{\mc{\O}}$, we will show that $p^*\mc{\O}^c$,
is contained in the complement $\mc{p^*\O}^c$ with respect to $p^*\O$.
Since $\mc{\O} \subseteq \O$ then $p^*\mc{\O} \subseteq p^*\O$, therefore we can consider $p^*\mc{\O}^c\coloneqq p^*\O \setminus p^*\mc{\O}$.
Let $K\to H$ be a transfer in $p^*\mc{\O}^c$, then we seek to show that $K\to H$ is in $\mc{p^*\O}^c$. 
As $K\to H\in p^*\O$, \cref{lem:in inflation iff restriction} implies that $KN\to HN\in p^{-1}\O$ and $KN\cap H=K$.
Similarly, as $K\to H \not \in p^*\mc{\O}$, \cref{lem:in inflation iff restriction}  implies that $KN\to HN\not \in p^{-1}\mc{\O}$.
\cref{cor:p^* in the interval} lets us apply the formula for $\mc{\O}^c$ to the interval $[N,G]$, to establish the existence of some $I < J$ in the interval $[N,HN]$ such that $I\to J$ is not in $p^{-1}\O$.
We illustrate this in the left diagram of \eqref{eqn:functoriality:complement}.
We then take intersections with $H$ and label the resulting transfers.
\begin{equation}\label{eqn:functoriality:complement}
\begin{tikzcd}[column sep = tiny, row sep = tiny]
	& HN &&&&& H \\
	{\textcolor{white}{J}KN\textcolor{white}{J}} && J & {} && {\textcolor{white}{J} K \textcolor{white}{J}} && {J \cap H} \\
	& {KN \cap J} && {\textcolor{white}{J} \, I \, \textcolor{white}{JJ}} &&& {K \cap J} && {I \cap H} \\
	&& {KN \cap I} &&&&& {K \cap I}
	\arrow["{p^{-1}\O}", from=2-1, to=1-2]
	\arrow["{- \cap H}", 
    from=2-4, to=2-6]
	\arrow["{p^*\O}", from=2-6, to=1-7]
	\arrow["{p^{-1}\O}", from=3-2, to=2-3]
	\arrow["{\not\exists p^{-1}\O}"', color={rgb,255:red,214;green,92;blue,92}, dashed, from=3-4, to=2-3]
	\arrow["a", from=3-7, to=2-8]
	\arrow["d"', dashed, from=3-9, to=2-8]
	\arrow["{p^{-1}\O}", from=4-3, to=3-2]
	\arrow["{p^{-1}\O}"', from=4-3, to=3-4]
	\arrow["b", from=4-8, to=3-7]
	\arrow["c"', from=4-8, to=3-9]
\end{tikzcd}
\end{equation}
We claim that $a,b,c$ are all in $p^*\O$ and $d$ is not.
The transfers $a$ and $c$ are each a restriction of the transfer $K\to H$.
Similarly, $b$ is a restriction of $KN\cap I\to KN \cap J$.
For $d$, we use \cref{lem:in inflation iff restriction} and seek to show that $(I\cap H)N \to (J\cap H)N$ is not in $p^{-1}\O$.
Applying \cref{lem:group product containment} twice reveals $(I\cap H)N \to (J\cap H)N$ is equal to $I\to J$.
Since $(I \cap H)N \to (J\cap H)N \notin p^{-1}\O$, then by \cref{lem:in inflation iff restriction} we have that $I\cap H \to J\cap H \notin p^*\O$.
Therefore $K\to H$ restricts into a compatibility failure, and $K\to H \in \mc{p^*\O}^c$.\end{proof}

This result then lets us extend \cref{prop: inflation preserves saturated compatibility} to the case when $\Om$ is not saturated.

\begin{corollary}[\cref{thm:inflation}(3)]   \label{prop: inflation preserves compatibility}

Let $N\trianglelefteq G$ be a normal subgroup. Let $\O,\O_m$ be two $G/N$-transfer systems, and $p:G\onto G/N$ the projection.
If $(\O,\O_m)$ is compatible, then so is $(p^*\O,p^*\O_m)$.
\end{corollary}

\begin{proof}
Recall that $(\O,\Om)$ is compatible if, and only if, $\O_m \subseteq \mc{\O}$ (\cref{prop:Maximal compatible characterises all compatibility}).
Thus as $p^*$ is monotone (\cref{lem:inflation monotone}), and  preserves maximal compatibility (\cref{prop: inflation preserves max compat}), this implies that $p^*\O_m\subseteq p^*\mc{\O}= \mc{p^*\O}$.
Which implies $(p^*\O, p^*\O_m)$ is compatible.
\end{proof}

\subsection{Consequence and Counterexamples}\label{ssec:misc inflation} 

In this section we discuss a powerful consequence of \cref{thm:inflation}, namely \cref{cor:universal-transfer}. This uses the right adjoint to the inflation functor, the \(N\)-fixed points.

\begin{definition}[{\cite[Section 5.1]{RubinOperadicLifts}, \cite[Appendix B]{BHOperads}}]
\label{def:pushforward}
If \(\O\) is a \(G\)-transfer system, let $\O^N$ be the $G/N$-transfer system given by restricting \(\O\) to \(\Sub(G/N)\subseteq \Sub(G)\).
\end{definition}

Rubin showed that the \(N\)-fixed points is right adjoint to the pull back \cite{RubinOperadicLifts}, so for any normal subgroup \(N\), we always have a map of \(G\)-transfer systems
\[
 p^\ast \O^N\to \O,
\]
and when we look at the part of \(\Sub(G)\) coming from \(\Sub(G/N)\), this map is just the identity.

Recall now that for any \(\O\), there is a minimal subgroup \(\NO\) which transfers to \(G\). Minimality guarantees also that \(\NO\) is normal.

\begin{lemma}\label{lem:no universal transfer means image of unit}
For a disklike $G$-transfer system $\O$, we have 
\[
p^\ast\O^N\cong \O.
\]
\end{lemma}

\begin{proof}
Since $\O$ is disklike, it is generated by transfers of the form $H\to G$. The minimality of \(\NO\) shows that any such \(H\) contains \(\NO\), so all of the generating transfers already sit in the image of \(\Sub(G/N)\) in \(\Sub(G)\).
\end{proof}

A consequence of \cref{thm:inflation}(4) and \Cref{lem:in inflation iff restriction} is that 
\[
\mc{p^*\O^N} = p^*\mc{\O^N}=\Res(p^{-1}\mc{\O^N}).
\]
Combining these with \cref{lem:no universal transfer means image of unit} gives a significant simplification in the computation of the maximal compatible transfer system for a disklike transfer system.

\begin{corollary}\label{cor:universal-transfer}
Let $\O$ be a disklike $G$-transfer system and let $p:G\to G/N_\O$ be the associated quotient map.
Then, the computation of $\mc{\O}$ reduces to the computation of $\mc{\O^N}$ in $\sub(G/N_\O)$:
\[
\mc{\O} = \Res(p^{-1}\mc{\O^N}).
\]
\end{corollary}

This simplification is powerful. 
For example, we can use this corollary to simplify the computation of the maximal compatible transfer system for the $C_{q^3r^3}$-transfer system of \cref{ex:inflation example}. 
We observe that $\NO = C_{qr}$, and that our computation reduces to the $C_{q^2r^2}$-transfer system of \cref{ex:algorithm example}.
In general, the smaller the interval $[N_\O,G]$ is in comparison to all of $\sub(G)$, the more effective this simplification is.

\begin{example}
\label{ex:universal transfer corollary not true in general}

The naive analogue of \cref{cor:universal-transfer} does not hold in general for categorical transfer systems.
Consider the lattice $P$ and transfer system $\O$ in  \cref{ex:categorical disk-like transfer system is not good}, which we reproduce here with additional labeling.
Recall that the bold blue transfer, $r'$, is the only non-trivial transfer in $\mc{\O}$. 
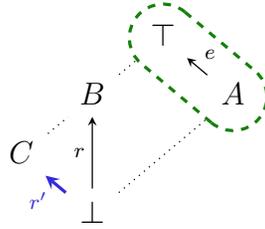
\[
\begin{tikzcd}[row sep = tiny, column sep = tiny, every matrix/.append style= {name=r}, remember picture]
	&& \top \\
	& B && A \\
	C \\
	& \bot
	\arrow[dotted, no head, from=2-2, to=1-3]
	\arrow["e"', from=2-4, to=1-3]
	\arrow[dotted, no head, from=3-1, to=2-2]
	\arrow["r", from=4-2, to=2-2]
	\arrow[dotted, no head, from=4-2, to=2-4]
	\arrow["{r'}", from=4-2, to=3-1, very thick, cblue]
\end{tikzcd}
   \begin{tikzpicture}[remember picture, overlay]
   \myroundpoly[ao,very thick, dashed]{r-2-4,r-1-3}{0.45cm};
   \end{tikzpicture}\]
The only disklike generator is $e$, therefore $N_\O = A$.
If the naive analogue of \cref{cor:universal-transfer} held, we should be able to compute $\mc{\O}$ by restricting our focus to the interval $[A,\top]$ (circled by a green dashed line), then applying $\Res$.
 The restriction of $\O$ to that interval is the complete transfer system on $[A,\top]$ (with unique nonreflexive transfer $e$), and is hence self-compatible.
 However, in $P$, $e$ restricts into the saturation failure formed by $r$ and $r'$ and hence $e\notin \mc{\O}$.
 The question of when an analogue of \cref{cor:universal-transfer} holds in the categorical setting remains open.
\end{example}

\bibliographystyle{amsalpha}
\let\oldaddcontentsline\addcontentsline
\renewcommand{\addcontentsline}[3]{}
\bibliography{bibby.bib}
\let\addcontentsline\oldaddcontentsline

\end{document}